\documentclass[12pt]{amsart}
\usepackage{amsmath,bbm, color}
\usepackage{latexsym,mathrsfs,bm}
\usepackage{mathtools}
\mathtoolsset{showonlyrefs}

\usepackage{url}
\usepackage[english]{babel}
\usepackage{paralist}
\usepackage{amsthm,amsfonts,amssymb,amsmath}
\usepackage[latin1]{inputenc}

\newcommand{\tRe}{\textup{Re }}
\newcommand{\tIm}{\textup{Im }}

\newcommand{\sumstar}{\sideset{}{^*}\sum}
\newcommand{\sumflat}{\sideset{}{^\flat}\sum}

\newcommand{\sumd}{\sideset{}{^d}\sum}

\newcommand{\cal}[1]{\mathcal{#1}}

\newcommand{\sym}{\textup{sym}}

\newtheorem{thm}{Theorem}[section]

\newtheorem{prop}[thm]{Proposition}
\newtheorem{lem}[thm]{Lemma}

\theoremstyle{remark}
\newtheorem{rem}{Remark}

\newtheorem{rem*}{Remark}

\newcommand{\sgn}{\operatorname{sgn}}

\newcommand{\bfrac}[2]{\left(\frac{#1}{#2}\right)}

\def\sumstar{\operatornamewithlimits{\sum\nolimits^*}}

\newcommand{\sumtwo}{\operatorname*{\sum\sum}}

\newcommand{\comment}[1]{}
\newcommand{\cL}{\mathcal{L}}

\newcommand{\cF}{\mathcal{F}}
\newcommand{\cG}{\mathcal{G}}

\let\originalleft\left
\let\originalright\right
\renewcommand{\left}{\mathopen{}\mathclose\bgroup\originalleft}
\renewcommand{\right}{\aftergroup\egroup\originalright}

\numberwithin{equation}{section}

\begin{document}
\title{Moments of quadratic twists of modular $L$-functions}

\date{\today}


\author[X. Li]{Xiannan Li}
\address{Mathematics Department \\ Kansas State University \\ Manhattan, KS 66503}
\email{xiannan@math.ksu.edu }

\allowdisplaybreaks
\numberwithin{equation}{section}
\selectlanguage{english}
\begin{abstract}
We prove an asymptotic for the second moment of quadratic twists of a modular $L$-function.  This was previously known conditionally on GRH by the work of Soundararajan and Young \cite{SY}.
\end{abstract}

\maketitle  
\section{Introduction}
\subsection{Background and statement of results}
Moments of $L$-functions are central objects of study within analytic number theory.  Generally, moments contain information about the distribution of values of $L$-functions and thus are related to a multitude of arithmetic objects.  One particularly interesting family is that of quadratic twists of modular $L$-functions.  This family is studied for its own interest and for applications to elliptic curves and coefficients of half integer weight modular forms.

To be more precise, let $f$ be a modular form of weight $\kappa$ for the full modular group and suppose that $f$ is a Hecke eigenform.  The results we describe below may be extended to $f$ of arbitrary level with some minor technical modifications.  The $L$-function associated with $f$ is given by
$$L(s, f) = \sum_n \frac{\lambda_f(n)}{n^s} = \prod_p \left(1-\frac{\lambda_f(p)}{p^s} + \frac{1}{p^{2s}}\right)^{-1},
$$for $\tRe s > 1$, and can be analytically continued to the entire complex plane.  Here $\lambda_f(n)$ are the Hecke eigenvalues with $\lambda_f(1) = 1$.  The completed $L$-function is given by
\begin{equation}
\Lambda_f(s, f) = \bfrac{1}{2\pi}^s \Gamma\left(s+\frac{\kappa-1}{2}\right) L(s, f),
\end{equation}satisfies the functional equation
\begin{equation}
\Lambda_f(s, f) = i^\kappa \Lambda(1-s, f).
\end{equation}  

For $d$ a fundamental discriminant, let $\chi_d(\cdot) = \bfrac{d}{\cdot}$ denote the primitive quadratic character with conductor $|d|$.  Then $f\otimes \chi_d$ is a primitive Hecke eigenform of level $|d|^2$, with $L$-function given by
\begin{equation}
L(s, f\otimes \chi_d) = \sum_n \frac{\lambda_f(n)\chi_d(n)}{n^s} = \prod_p \left(1-\frac{\lambda_f(p)\chi_d(p)}{p^s} + \frac{\chi_d(p)^2}{p^{2s}}\right)^{-1}
\end{equation}for $\tRe s > 1$.  The completed $L$-function is
\begin{equation}
\Lambda(s, f \otimes \chi_d) = \bfrac{|d|}{2\pi}^s \Gamma\left(s+\frac{\kappa-1}{2}\right) L(s, f\otimes \chi_d),
\end{equation}and satisfies the functional equation
\begin{equation}\label{eqn:funceqntwist}
\Lambda(s, f\otimes \chi_d) = i^\kappa \epsilon(d) \Lambda(1-s, f\otimes \chi_d),
\end{equation}where $\epsilon(d) = \bfrac{d}{-1} = \pm 1$ depending on the sign of $d$.  Note that if $i^\kappa \epsilon(d) = -1$, then $L(1/2, f\otimes \chi_d) = 0$.  

We let $\sumstar$ denote a sum over squarefree integers while $\sumflat$ will denote a sum over fundamental discriminants.  For convenience, we restrict the modulus to be of the form $8d$ where $d$ is squarefree; one can study other discriminants using the same methods.  In this context, it is of high interest to understand moments of the form
\begin{equation}
M(k) := \sumstar_{\substack{0<8d<X\\ (d, 2) = 1}} L(1/2, f\otimes \chi_{8d})^k.
\end{equation} Keating and Snaith \cite{KS} conjectured that
$$M(k) \sim C(k, f) X(\log X)^{\frac{k(k-1)}{2}},
$$for an explicit constant $C(k, f)$.  Unconditionally, this was known for the first moment $k=1$ from Iwaniec's work \cite{Iwaniec}.  Further, based on knowledge of the twisted first moment, Radziwill and Soundararajan \cite{RS} proved that $M(k) \ll X(\log X)^{\frac{k(k-1)}{2}}$ for $0\le k\le 1$.  

The case $k=2$ has proved more challenging.  The work of Heath-Brown \cite{HB} implies that $M(2) \ll X^{1+\epsilon}$.  Assuming the Generalized Riemann Hypothesis (GRH), the work of Soundararajan \cite{Smoments} implies that $M(2) \ll X(\log X)^{1+\epsilon}$.  Based on ideas from \cite{Smoments}, Soundararajan and Young \cite{SY} proved the conjectured asymptotic conditionally, assuming GRH.  Our main result below addresses this unconditionally.

\begin{thm}\label{thm:main}
With notation as above, and for $\kappa \equiv 0 \bmod 4$
\begin{equation}
\sumstar_{\substack{0<8d<X\\ (d, 2) = 1}} L(1/2, f\otimes \chi_{8d})^2 \sim C_f X\log X, 
\end{equation}
where
\begin{equation}
C_f = \frac{2}{\pi^2} L(1, \sym^2 f)^3 \cal H_2(0, 0),
\end{equation}where $\cal H_2$ is defined as in Lemma \ref{lem:calG}.
\end{thm}

If we include a smooth weight in the sum over $d$ above, the result can be proven with an error term of quality $O(X(\log X)^{1/2 + \epsilon})$ and improved to $O(X(\log X)^\epsilon)$ with a little effort.  This is made explicit in \S \ref{sec:thmproof}.  Of course, we can prove Theorem 1 with a saving of a power of $\log$ with some care in the choice of smoothing function.

Our techniques extend to give the expected asymptotic for the fourth moment of quadratic Dirichlet $L$-functions unconditionally.  For this family, the first and second moments were computed by Jutila \cite{Jutila}, and the third moment by Soundararajan \cite{Snonvanishing}.  There were a number of refinements with improved error terms: on the first  \cite{Y1st} and third moments \cite{Y3rd} by Young, the second moment by Sono \cite{Sono} using similar methods, and a further refinement of the third moment by Diaconu and Whitehead \cite{DiaconuWhitehead} explicating a power saving secondary term.  The fourth moment was computed recently again assuming GRH by Shen \cite{Shen}, following the approach of Soundararajan and Young \cite{SY}.  We also mention the recent work of Florea \cite{Florea}, which gives the expected asymptotic for the analogous fourth moment over the function field $\mathbb{F}_q[x]$ (where the Riemann hypothesis is known) with the base field $\mathbb{F}_q$ fixed and genus going to infinity.  

As mentioned before, this family of $L$-functions has received special scrutiny because of its connections to elliptic curves and half integer weight modular forms.  Let $m_d$ be the order of vanishing of $L(s, f\otimes \chi_{8d})$ at $s = 1/2$.  In the case when $f$ corresponds to an elliptic curve, the Birch and Swinnerton-Dyer conjecture relates $m_d$ to the rank of the twisted elliptic curve.  

For ease of notation, let
$$R(X) = \sumstar_{\substack{0<8d<X\\ (d, 2) = 1}} m_d.
$$
Goldfeld \cite{Go} proved that $R(X) \ll X$ conditionally on GRH.  Trivially, $R(X) \ll X \log X$, while the work of Perelli and Pomykala \cite{PP} gives the refined bound $R(X) = o(X \log X)$.  Our methods yield $R(X) \ll X \log \log X$ proceeding along the same lines; see Theorem 5 of \cite{PP} for more details.


\subsection{Rough concept}\label{sec:roughconcept}
We now briefly discuss the main ideas in the proof.  In the rest of the paper, we let $\bfrac{m}{n}$ denote the usual Kronecker symbol.  After an application of the approximate functional equation, we morally need to understand sums like
\begin{equation}
\sumstar_{m\asymp X} \left|\sum_{n\ll X} a(n) \bfrac{m}{n} \right|^2,
\end{equation}where $a(n) = \frac{\lambda_f(n)}{\sqrt{n}}$.  Standard tools like the functional equation and Poisson summation are not useful in this range, but become useful in the easier range
\begin{equation}\label{eqn:truncated}
\sumstar_{m\asymp X} \left|\sum_{n\ll X/(\log^A X)} a(n) \bfrac{m}{n} \right|^2,
\end{equation}for some large $A>0$.  Thus, the challenge is to bound sums of the form
\begin{equation}
S = \sumstar_{m\asymp X} \left|\sum_{n\asymp N} a(n) \bfrac{m}{n} \right|^2,
\end{equation}
when $N$ is close to $X$.  The influential work of Heath-Brown \cite{HB} implies that $S \ll X^{1+\epsilon}$, but we need a bound as strong as $S \ll X(\log X)^\delta$ for $\delta < 1$ for our application.  In fact we will show that  
\begin{equation}\label{eqn:sketchbdd}
\sumstar_{m\asymp X} \left|\sum_{n\asymp N} a(n) \bfrac{m}{n} \right|^2 \ll X,
\end{equation} which is best possible up to the implied constant.  Assuming \eqref{eqn:sketchbdd}, dyadic summation for $\frac{X}{\log^A X} \leq N \ll X$ gives the bound  
$$\sumstar_{m\asymp X} \left|\sum_{X/(\log^A X)\ll n\ll X} a(n) \bfrac{m}{n} \right|^2 \ll X (\log X)^\epsilon,
$$ whence it suffices to study the easier quantity in \eqref{eqn:truncated}.  

This type of truncation strategy appeared in the work of Soundararajan \cite{Sfourth}, Soundararajan and Young \cite{SY}, and some later papers.  The main difficulty is proving the bound \eqref{eqn:sketchbdd}.  Indeed the new content in the work of Soundararajan and Young \cite{SY} was the implicit proof that
$$S \ll X(\log X)^{1/2+\epsilon}$$ conditionally on GRH.  To be more precise, Soundararajan and Young do not explicitly state this bound but rather proceed via conditional bounds on shifted moments instead, which is in turn based on important ideas from the work of Soundararajan in \cite{Smoments}. 

Since the proof of \eqref{eqn:sketchbdd} is the novel part of this work, we now give a sketch of the approach.  For simplicity, suppose that $N = X$.  Now, fix a large parameter $L$, and write for a prime $p \asymp \sqrt{L}$
\begin{align}\label{eqn:Sintro}
S = \sumstar_{m\asymp X} \left|\sum_{n\asymp X} a(n) \bfrac{m}{n} \right|^2
&= \sumstar_{m\asymp X} \left|\sum_{\substack{n\asymp X\\ p\nmid n}} a(n)\bfrac{m p^2}{n} + \sum_{\substack{n\asymp X \\ p|n}} a(n) \bfrac{m}{n} \right|^2.
\end{align}
We may use the Hecke relation to handle the second sum.  For this sketch, we focus on the more illustrative first sum.  Letting $\cal P (L) \asymp \frac{\sqrt{L}}{\log L}$ be the number of primes in the interval $[\sqrt{L}, 2\sqrt{L}]$, we sum over all $p \in [\sqrt{L}, 2\sqrt{L}]$ to see that
\begin{align}
\cal P (L) S 
&\ll \sum_{\sqrt{L} \le p \le 2\sqrt{L}} \sumstar_{m\asymp X} \left|\sum_{n\asymp X} a(n) \bfrac{m p^2}{n}\right|^2 + \textup{other} \notag \\
&\le \sum_{\substack{m \asymp 4XL}} \left|\sum_{n\asymp X} a(n) \bfrac{m}{n}\right|^2 + \textup{other}. \notag
\end{align}
Here, we have used positivity and the fact that when $m_1$ and $m_2$ are squarefree, 
\begin{equation}\label{eqn:unique}
m_1 p_1^2 = m_2 p_2^2
\end{equation}only when $p_1 = p_2$ and $m_1 = m_2$.  We have embedded our original sum over $m$ into a longer sum, so that it is now advantageous to execute Poisson over $m$.  Note that discarding the squarefree condition on $m$ can be disastrous for arbitrary coefficients $a(n)$.\footnote{For instance, if $a(n) = \frac{1}{\sqrt{n}}$, the contribution from the square values of $m$ alone gives a contribution $\gg \sqrt{ML} X$.}  We therefore expect to crucially use the special properties of $a(n) = \frac{\lambda_f(n)}{\sqrt{n}}$.

Opening up the square and applying Poisson summation roughly gives that
\begin{align*}
&\sum_{\substack{m \asymp XL}} \left|\sum_{n\asymp X} a(n) \bfrac{m}{n}\right|^2\\
&= \cal C_f XL + \frac{XL}{2} \sum_{\substack{n_1, n_2 \asymp X }} \frac{\lambda_f(n_1) \lambda_f(n_2)}{\sqrt{n_1 n_2} n_1 n_2} \sum_{k \neq 0} (-1)^k G_{k}(n_1n_2) W\bfrac{kXL}{2n_1 n_2},
\end{align*}for some constant $\cal C_f$ depending only on $f$, $W$ a smooth function with rapid decay, and where the Gauss-like sum $G_{k}(n_1n_2)$ is defined in \eqref{eqn:Gdef}.  The sum over $k$ is essentially restricted to $k\ll X^2/XL \asymp X/L$, so we need to bound
\begin{align}\label{eqn:Gside}
\frac{XL}{2} \sum_{k\asymp X/L} \sum_{\substack{n_1, n_2 \asymp X }} \frac{\lambda_f(n_1) \lambda_f(n_2)}{\sqrt{n_1 n_2} n_1 n_2} G_{k}(n_1n_2).
\end{align}Now we replace $G_{k}(n_1n_2)$ by $\chi_{k}(n_1n_2) \sqrt{n_1n_2}$, which is generically true for $n_1 n_2$ squarefree, and restrict our attention to squarefree $k$, so we hope to replace \eqref{eqn:Gside} by a quantity like
\begin{align}\label{eqn:quadcharside}
\frac{XL}{X} \sumstar_{k\asymp X/L} \left|\sum_{\substack{n\asymp X }} \frac{\lambda_f(n) \chi_k(n)}{\sqrt{n}} \right|^2,
\end{align}where we have replaced a factor $\frac{1}{\sqrt{n_1n_2}}$ by its size of $\frac 1X$.  Since the conductor $k\asymp X/L$ has been reduced, it now makes sense to apply the functional equation of $L(s, f\otimes \chi_k)$ to transform the sum over $n$ to a sum of length $X/L^2$, which is shorter again than the length of the sum over $k$ by a factor of $L$.  This suggests that we should succeed if we continue this procedure by iteratively applying Poisson over $k$ and the the functional equation over $n$.  

The use of prime factors to inflate a sum in the context of a large sieve appeared in the work of Forti and Viola \cite{FV} and notably in the work of Heath-Brown \cite{HB}.  Our specific coefficients $a(n) = \frac{\lambda_f(n)}{\sqrt{n}}$ enboldens us to use prime squares, thereby discarding primitivity in our character sum.  The choice of prime squares has a number of advantages.  Indeed, to preserve our coefficients, it is important that $\bfrac{p^2}{n}$ is typically trivial.  Moreover, the uniqueness property from \eqref{eqn:unique} allows us to avoid counting multiplicities and thus avoids losing factors of $\log X$.  Here, the squarefree condition on $\sumstar_m$ helps rather than hinders.  It guarantees uniqueness in \eqref{eqn:unique} and we can entirely discard the squarefree condition in our situation when convenient.  One last important property we use is that there are a large number of prime squares - the fact that the number of primes in the interval $[\sqrt{L}, 2\sqrt{L}]$ is large serves to control the loss of constant factors which accompany our arguments.  This is crucial in our inductive step.

In this rough sketch, we have oversimplified many parts of the proof.  One place which is particularly egregious is the replacement of \eqref{eqn:Gside} by \eqref{eqn:quadcharside}, since this glosses over technical complications and hides an important structural feature.  To see this, we expect parts of \eqref{eqn:quadcharside} to resemble
\begin{align*}
L\sumstar_{k\asymp X/L} \left|\sum_{\substack{n\asymp X/L^2 }} \frac{\lambda_f(n) \chi_k(n)}{\sqrt{n}} \right|^2
 = C X \sum_{\substack{n_1, n_2 \asymp X/L^2 \\n_1n_2 = \square }} \frac{\lambda_f(n_1) \lambda_f(n_2)}{\sqrt{n_1n_2}} + \textup{smaller term},
\end{align*}for some constant $C$.  In other words, the "diagonal" contribution arising from the terms when $n_1n_2$ is a perfect square dominates.  However, generically $G_k(n_1n_2) = 0$ when $n_1n_2$ is not squarefree, so that the same "diagonal" contribution simply does not exist in the sum \eqref{eqn:Gside}.  This is one of the underlying reason why it requires care and dexterity to avoid losing factors of $\log X$.  In particular, careful analysis of the factors at prime squares and higher powers is crucial.  We refer the reader to \S \ref{sec:propproof} for a more accurate picture.

Since we aim to prove the optimal bound \eqref{eqn:sketchbdd}, there are some uncommon features in our proof.  For instance, in order to control constants which depend on smooth test functions, we prove our main Proposition \ref{prop:key} only for fixed smooth functions $F$ and $G$.  These functions need to be chosen with some care in Lemma \ref{lem:testfunc} and around \eqref{eqn:G}.  In particular, the fact that $\hat{F}$ is compactly supported and that $G$ may be used to form a dyadic partition of unity is quite useful in the proof.

In \S \ref{sec:prelim}, we gather some basic results, and in \S \ref{sec:prop}, we state the main Propositions and provide an outline of the rest of the paper.
\\\\
{\bf Acknowledgement.} I would like to thank J. Stucky for a number of helpful editorial remarks.  I would also like to thank N. Ng and M. Young for helpful editorial comments.  This work was partially supported by a Simons Foundation Collaboration Grant (524790).

\section{Preliminary results}\label{sec:prelim}
Here we gather some basic tools.  First, we have the standard approximate functional equation.

\begin{lem}\label{lem:approximatefunctionalequation}
For $d$ a fundamental discriminant, 
\begin{equation}
L(s, f\otimes\chi_d) = \cal{A}(s, d) + i^\kappa \epsilon(d) \bfrac{|d|}{2\pi}^{1-2s} \frac{\Gamma(1-s)}{\Gamma(s)} \cal{A}(1-s, d),
\end{equation}where
\begin{equation}
\cal{A}(s, d) = \sum_{n\ge 1} \frac{\lambda_f(n) \chi_d(n)}{n^s} W_s\bfrac{n}{|d|}
\end{equation}and for any $c>0$,
\begin{equation}
W_s(x) = \frac{1}{2\pi i} \int_{(c)} \frac{\Gamma(s+\frac{\kappa-1}{2}+w)}{\Gamma(s+\frac{\kappa-1}{2})} (2\pi x)^{-w} \frac{dw}{w}.
\end{equation}
\end{lem}

We refer the reader to Theorem 5.3 of the Iwaniec and Kowalski's book \cite{IK} for the proof of Lemma \ref{lem:approximatefunctionalequation}.  Now, we define the Gauss like sum
\begin{equation} \label{eqn:Gdef}
G_k(n) = \left( \frac{1-i}{2} + \bfrac{-1}{n} \frac{1+i}{2} \right) \sum_{a \bmod n} \bfrac{a}{n} e\bfrac{ak}{n}.
\end{equation}

The sum $G_k(n)$ appeared in the work of Soundararajan \cite{Snonvanishing} and we record Lemma 2.3 from Soundararajan \cite{Snonvanishing} below.
\begin{lem}\label{lem:G}
For $m, n$ relatively prime odd integers, $G_k(mn) = G_k(m)G_k(n)$, and for $p^\alpha \| k$ (setting $\alpha = \infty$ for $k=0$), then
\begin{equation}
G_k(p^\beta) = 
\begin{cases}
0, &\textup{if $\beta \le \alpha$ is odd}, \\
\phi(p^\beta),  &\textup{if $\beta \le \alpha$ is even}, \\
-p^\alpha,  &\textup{if $\beta = \alpha + 1$ is even}, \\
\bfrac{kp^{-\alpha}}{p} p^\alpha \sqrt{p},  &\textup{if $\beta = \alpha+1$ is odd}, \\
0,  &\textup{if $\beta \ge \alpha+2$}.
\end{cases}
\end{equation}
\end{lem}


As alluded to in \S \ref{sec:roughconcept}, $G_k(n)$ appears when applying Poisson summation as in Lemma \ref{lem:PoissonSummation} below.

\begin{lem}\label{lem:PoissonSummation}
Let $F$ be a Schwartz class function over the real numbers and suppose that $n$ is an odd integer.  Then
\begin{equation}\label{eqn:Poissonsimple}
\sum_{d} \bfrac{d}{n} F\bfrac{d}{Z}
= \frac{Z}{n} \sum_{k\in \mathbb{Z}} G_k(n) \check{F}\bfrac{kZ}{n},
\end{equation}and
\begin{equation}\label{eqn:Poisson8d}
\sum_{(d, 2)=1} \bfrac{d}{n} F\bfrac{d}{Z}
= \frac{Z}{2 n} \bfrac{2}{n} \sum_{k\in \mathbb{Z}} (-1)^k G_k(n) \check{F}\bfrac{kZ}{2n},
\end{equation}
where $G_k(n)$ is defined as in \eqref{eqn:Gdef}, and the Fourier-type transform of $F$ is defined to be
\begin{equation}
\check{F}(y)=\int_{-\infty}^{\infty} (\cos(2\pi xy) + \sin(2\pi xy)) F(x) dx.
\end{equation}

Further for $F$ even and $y \neq 0$,
\begin{align}\label{eqn:mellincheckevenF}
\check{F}(y) 
&= 2\int_0^\infty F(x) \cos (2\pi xy) dx\\
&= \frac{2}{2\pi i} \int_{(1/2)} \tilde{F}(1-s) \Gamma(s) \cos \bfrac{\pi s}{2} (2\pi |y|)^{-s} ds,
\end{align}while for $F$ supported on $[0, \infty)$,
\begin{align}\label{eqn:mellincheckpositiveF}
\check{F}(y) = \frac{1}{2\pi i} \int_{(1/2)} \tilde{F}(1-s) \Gamma(s) (\cos + \sgn(y) \sin) \bfrac{\pi s}{2} (2\pi |y|)^{-s} ds,
\end{align}where 
$$\tilde{F}(s) = \int_0^\infty F(x) x^s \frac{dx}{x}
$$ 
is the usual Mellin transform of $F$.

\end{lem}
\begin{proof}
The first assertions in \eqref{eqn:Poissonsimple} and \eqref{eqn:Poisson8d} are contained in the proof of Lemma 2.6 of \cite{Snonvanishing}.  The assertion in \eqref{eqn:mellincheckevenF} and \eqref{eqn:mellincheckpositiveF} follows by Mellin inversion, and we refer the reader to \S 3.3 of \cite{SY} for details.
\end{proof}

For $F$ a Schwartz class function, we write the usual Fourier transform of $F$ as
\begin{equation}
\hat{F} (y) = \int_{-\infty}^{\infty} e(-xy) F(x) dx.
\end{equation}Note that $\check{F}(x) = \frac{1+i}{2} \hat F(x) + \frac{1-i}{2} \hat F(-x)$, and if $F$ is even then $\check{F} = \hat{F}$.

Applying Lemma 2.3 gives rise to the "diagonal" contribution corresponding to $k=0$ and the off-diagonal contribution.  For convenience, we record some further calculations here.

\begin{lem}\label{lem:Poisson2}
Let $H(x, y, z)$ be a Schwartz class function on $\mathbb{R}^3$, $H_1(y, z) = \int_{-\infty}^\infty H(x, y, z) dx$, and 
$$\tilde{H}(s, u, v) = \int_0^\infty \int_0^\infty \int_0^\infty H(x, y, z) x^s y^u z^v \frac{dx}{x} \frac{dy}{y} \frac{dz}{z}.
$$Let $n_1$ and $n_2$ be any odd positive integers.  Then
 \begin{align}\label{eqn:Poisson2simple}
\sum_{d} \bfrac{d}{n_1 n_2} H\left( \frac{d}{X}, n_1, n_2 \right) 
&= \delta_{\square}(n_1n_2) X H_1(n_1, n_2) \prod_{p|n_1n_2} \left(1-\frac 1p\right) \\
&+  X \sum_{\substack{k\in \mathbb{Z}\\ k \neq 0}} \frac{G_k(n_1 n_2)}{n_1 n_2} I(k, n_1, n_2),
\end{align}and 
\begin{align}\label{eqn:Poisson2}
\sum_{(d, 2)=1} \bfrac{8d}{n_1 n_2} H\left( \frac{d}{X}, n_1, n_2 \right) 
&= \delta_{\square}(n_1n_2) \frac{X}{2} H_1(n_1, n_2) \prod_{p|n_1n_2} \left(1-\frac 1p\right) \\
&+  \frac{X}{2} \sum_{\substack{k\in \mathbb{Z}\\ k \neq 0}} (-1)^k \frac{G_k(n_1 n_2)}{n_1 n_2} I(k, n_1, n_2),
\end{align}
where $\delta_{\square}(n) = 1$ when $n$ is a perfect square and vanishes otherwise.  Moreover, if $H(x, y, z)$ is supported on $\mathbb{R}_+^3$, then
$$I(k, n_1, n_2) = \frac{1}{(2\pi i)^3} \int_{(\epsilon)}\int_{(\epsilon)}\int_{(\epsilon)} \tilde{H}(1-s, u, v) n_1^{-u} n_2^{-v} \bfrac{n_1n_2}{\pi X |k|}^s \Gamma(s) (\cos + \sgn(y) \sin) \bfrac{\pi s}{2}  du dv ds
$$and if $H(x, y, z)$ is supported on $\mathbb{R} \times \mathbb{R}_+^2$, with $H$ even in $x$, then
$$I(k, n_1, n_2) = \frac{2}{(2\pi i)^3} \int_{(\epsilon)}\int_{(\epsilon)}\int_{(\epsilon)} \tilde{H}(1-s, u, v) n_1^{-u} n_2^{-v} \bfrac{n_1n_2}{\pi X |k|}^s \Gamma(s) \cos \bfrac{\pi s}{2}  du dv ds
$$

\end{lem}
\begin{proof}
For $n_1n_2$ odd, $\bfrac{8d}{n_1 n_2} = \bfrac{2}{n_1n_2} \bfrac{d}{n_1n_2}$.  The Lemma follows upon applying Lemma \ref{lem:PoissonSummation} to the left side of \eqref{eqn:Poisson2simple} and \eqref{eqn:Poisson2} respectively and taking Mellin transforms in other variables.  The Lemma follows upon noting that $G_0(n_1n_2) = \phi(n_1n_2)$ when $n_1n_2$ is a perfect square, and vanishes otherwise.

\end{proof}

After applying Lemmas \ref{lem:PoissonSummation} and \ref{lem:Poisson2}, we will be led to examine the Dirichlet series 
\begin{equation}\label{eqn:Zdef}
Z(\alpha, \beta, \gamma) = Z(\alpha, \beta, \gamma; k_1, q) = \sum_{k_2 \ge 1}\sumtwo_{\substack{(n_1, 2q) = 1 \\ (n_2, 2q) = 1}} \frac{\lambda_f(n_1) \lambda_f(n_2)}{n_1^\alpha n_2^\beta k_2^{2 \gamma}} \frac{G_{k_1 k_2^2}(n_1 n_2)}{n_1 n_2}.
\end{equation}
The Lemma below follows by examining local factors using Lemma \ref{lem:G}.  This is a slight refinement of Lemma 3.3 of \cite{SY}, and we provide a proof for the sake of completeness.  

\begin{lem}\label{lem:Z}
Let $k_1$ be squarefree.  Let $m= k_1$ if $k_1 \equiv 1 \bmod 4$ and $m = 4k_1$ for $k_1 \equiv 2, 3 \bmod 4$.  Then
\begin{align*}
Z(\alpha, \beta, \gamma) = L(1/2+\alpha, f\otimes \chi_{m}) L(1/2+\beta, f\otimes \chi_{m})Y (\alpha, \beta, \gamma; k_1),
\end{align*}
for
\begin{align*}
Y (\alpha, \beta,\gamma; k_1) = \frac{Z_2(\alpha, \beta, \gamma)}{\zeta(1+\alpha+\beta) L(1+2\alpha, \sym^2 f) L(1+\alpha+\beta, \sym^2 f) L(1+2\beta, \sym^2 f)},
\end{align*}
where $Z_2(\alpha, \beta, \gamma) = Z_2(\alpha, \beta, \gamma; k_1,q)$ is analytic in the region $\tRe \alpha, \beta \ge -\delta/2$ and $\tRe \gamma \ge 1/2 + \delta$ for any $0<\delta < 1/3$.  Moreover, in the same region, $Z_2(\alpha, \beta, \gamma) \ll d(q)$ where the implied constant may depend only on $\delta$ and $f$.
\end{lem}
\begin{proof}
By multiplicativity, we write
\begin{equation}
Z(\alpha, \beta, \gamma; k_1, q) = \prod_{p} \cF(p),
\end{equation}where
\begin{equation}
\cF(p) = \sum_{n_1, n_2, k_2 \ge 0} \frac{\lambda_f(p^{n_1}) \lambda_f(p^{n_2})}{p^{n_1 \alpha} p^{n_2 \beta} p^{2 k_2 \gamma}} \frac{G_{k_1 p^{2 k_2}}(p^{n_1+ n_2})}{p^{n_1+ n_2}} 
\end{equation}for $p\nmid 2q$, and 
\begin{equation}
\cF(p) = \left(1-\frac{1}{p^{2\gamma}}\right)^{-1},
\end{equation} for $p|2q$.  Let 
$$\cG(p, s) = \left(1-\frac{\lambda_f(p) \chi_{k_1}(p)}{p^{s}} + \frac{\chi_{k_1}(p)^2}{p^{2s}}\right)^{-1},$$
where $\chi_{k_1}(n) = \bfrac{k_1}{n} = \bfrac{m}{n}$ for all odd $n$.  Further write $L(s, \sym^2 f) = \sum_n \frac{A(n)}{n^s}$.  Now suppose that $\tRe \gamma \ge 1/2+\delta$ and $\tRe \alpha, \tRe \beta \ge -c$ for some $0<c < \delta<1/3$.  Then for $p|k_1$ and $p\nmid 2q$, we have
\begin{align*}
\cF(p) 
&= \sum_{k_2 \ge 0} \frac{1}{p^{2k_2 \gamma}} \left(\sum_{h=0}^{k_2} \frac{\phi( p^{2h} )}{p^{2h}} \sum_{\substack{i, j \\ i+j = 2h}} \frac{\lambda_f(p^i)\lambda_f(p^j)}{p^{i\alpha + j\beta}} - \frac{1}{p} \left( \sum_{\substack{i, j \\ i+j = 2k_2 +2}} \frac{\lambda_f(p^i)\lambda_f(p^j)}{p^{i\alpha + j\beta}} \right)   \right)\\
&= 1- \frac 1p \left(\frac{\lambda_f(p^2)}{p^{2\alpha}} + \frac{\lambda_f(p^2)}{p^{2\beta}} + \frac{\lambda_f(p)^2}{p^{\alpha+\beta}} \right) + O\bfrac{1}{p^{1+2\delta - 2c}}\\
&= \left(1 - \frac{A(p)}{p} \left(\frac{1}{p^{2\alpha}} + \frac{1}{p^{2\beta}}+\frac{1}{p^{\alpha+\beta}} \right) - \frac{1}{p^{1+\alpha+\beta}} + O\bfrac{1}{p^{3/2 - 3c}} \right)   + O\bfrac{1}{p^{1+2\delta - 2c}}
\end{align*}

When $p\nmid 2qk_1$, we have that
\begin{align*}
\cF(p) 
&= \sum_{k_2 \ge 0} \frac{1}{p^{2k_2 \gamma}} \left(\sum_{h=0}^{k_2} \frac{\phi( p^{2h} )}{p^{2h}} \sum_{\substack{i, j \\ i+j = 2h}} \frac{\lambda_f(p^i)\lambda_f(p^j)}{p^{i\alpha + j\beta}} + \frac{\chi_{k_1}(p)}{\sqrt{p}} \left( \sum_{\substack{i, j \\ i+j = 2k_2 +1}} \frac{\lambda_f(p^i)\lambda_f(p^j)}{p^{i\alpha + j\beta}} \right) \right)\\
&= 1 + \frac{\lambda_f(p) \chi_{k_1}(p)}{\sqrt{p}} \left( \frac{1}{p^{\alpha }} + \frac{1}{p^{\beta}} \right) + O\bfrac{1}{p^{1+2\delta - 2c}}\\
&= \cG(p,1/2+\alpha)\cG(p,1/2+\beta) \\
&\times \left(1 - \frac{A(p)}{p} \left(\frac{1}{p^{2\alpha}} + \frac{1}{p^{2\beta}}+\frac{1}{p^{\alpha+\beta}} \right) - \frac{1}{p^{1+\alpha+\beta}} + O\bfrac{1}{p^{3/2 - 3c}} \right)   + O\bfrac{1}{p^{1+2\delta - 2c}}.
\end{align*}
We now set $c \le \delta/2$ and note then that $c<1/6$, which in turn implies that
\begin{equation}
\prod_{p|2q} \left(1+ \frac{C_0}{p^{1/2 - c}} \right) \ll d(q)
\end{equation}for any constant $C_0$.  
\end{proof}

One of our basic tools will be to apply the functional equation directly.  This is done in the Lemma below.

\begin{lem}\label{lem:FE}
For $m$ a fundamental discriminant, and $G$ any Schwartz class function,
\begin{align*}
\sum_{n} \frac{\lambda_f(n) }{n^{1/2+z}} \bfrac{m}{n} G\bfrac{n}{N}  = \bfrac{2\pi}{|m|}^{2z} \sum_{n} \frac{\lambda_f(n)}{n^{1/2-z}} \bfrac{m}{n} \grave{G}_z\bfrac{4\pi^2 nN}{|m|^2}, 
\end{align*}where
\begin{equation}
\grave{G}(x) = \frac{1}{2\pi i} \int_{(2)} \frac{\Gamma\left(s - z + \frac{\kappa}{2}\right)}{\Gamma\left(- s+z+ \frac{\kappa}{2}\right)} x^{-s} \tilde{G}(-s) ds.
\end{equation}

\end{lem}
\begin{proof}
Let $c = |\tRe(z)| + 1$.  For $\tilde G$ the Mellin transform of $G$, we have
\begin{align*}
\sum_{n} \frac{\lambda_f(n) }{n^{1/2+z}} \bfrac{m}{n} G\bfrac{n}{N} 
=\frac{1}{2\pi i} \int_{(c)} L(1/2+z+s, f\otimes \chi_{m}) N^s \tilde{G}(s) ds.
\end{align*}Now, shifting the contour of integration to the line $\tRe s = -c$, applying the functional equation \eqref{eqn:funceqntwist} for $L(1/2+z+s, f\otimes \chi_{m})$ and a change of variables gives
\begin{align*}
&\sum_n \frac{\lambda_f(n) \chi_{m}(n)}{\sqrt{n}} G\bfrac{n}{N}  \\
&=\frac{1}{2\pi i} \int_{(-c)} \bfrac{|m|}{2\pi}^{-2(s+z)} \frac{\Gamma\left(\frac 12 - s - z + \frac{\kappa-1}{2}\right)}{\Gamma\left(\frac 12  + s +z+ \frac{\kappa-1}{2}\right)} L(1/2-s-z, f\otimes \chi_{m}) N^s \tilde{G}(s) ds \\
&=\frac{1}{2\pi i} \int_{(c)} \bfrac{|m|}{2\pi}^{2s-2z}  \frac{\Gamma\left(s -z +\frac{\kappa}{2}\right)}{\Gamma\left(-s + z + \frac{\kappa}{2}\right)}  L(1/2+s-z, f\otimes \chi_{m}) N^{-s} \tilde{G}(-s) ds \\
&= \bfrac{2\pi}{|m|}^{2z} \sum_n \frac{\lambda_f(n) \chi_{m}(n)}{n^{1/2-z}} \grave{G}\bfrac{4\pi^2 nN }{ |m|^2},  
\end{align*}as desired.
\end{proof}

For
$$g(s) = \frac{\Gamma\left(s + \frac{\kappa}{2}\right)}{\Gamma\left(-s + \frac{\kappa}{2}\right)},
$$and $s = \sigma + it$ with $\sigma > 0$, Stirling's formula implies that (see e.g. \S 5.A.4 \cite{IK})
\begin{equation}\label{eqn:Gammaestimate}
g(s) \ll (1+|t|)^{2\sigma}. 
\end{equation}
This gives the standard estimate
\begin{equation}\label{eqn:Fgravebdd}
\grave{G}_{it}(x) \ll_A \bfrac{(1+|t|)^2}{1+x}^A
\end{equation}for any $A > 0$ upon shifting contours to the right.

It will considerably simplify parts of our argument to use the test function discussed in the Lemma below.

\begin{lem}\label{lem:testfunc}
Let $c_0$ and $c_1$ be any fixed positive real numbers.  Then there exists a smooth non-negative even Schwartz class function $F$ such that $F(x) \ge 1$ for all $x \in [-c_1, c_1]$ and $\hat{F}(x)$ is even and compactly supported on $[-c_0, c_0]$.  It follows that $\check{F}(x)$ is also even and compactly supported on $[-c_0, c_0]$.
\end{lem}

\begin{proof}
We let $h_0$ be a smooth even non-negative function compactly supported on $[-c_0/2, c_0/2]$, and let $h = h_0*h_0$, so that $h$ is smooth, even, non-negative and supported on $[-c_0, c_0]$.  Let $g = \hat{h}$ so for $h_0$ not identically $0$, $h(0)>0$, and $g(0) > 0$ also.  Since $g$ is non-negative, even and Schwartz class, setting $F(x) = C_1 g(C_2 x)$ for some constants $C_1$ and $C_2 \le 1$ produces the desired function. 
\end{proof}

We now let $G$ be a smooth real-valued function compactly supported on $[3/4, 2]$ which satisfies
\begin{align}\label{eqn:G}
G(x) = 1  &\textup{ for all $x \in [1, 3/2]$ } \notag \\
G(x) + G(x/2) = 1  &\textup{ for all $x \in [1, 3]$}.
\end{align} This may be done by starting with $G(x)$ defined appropriately on $(-\infty, 3/2]$, and then letting $G(x) = 1-G(x/2)$ on $(3/2, 2]$, using that $G$ is already defined on $[3/4, 1]$.  Functions like $G$ appear in standard constructions of partitions of unity and we refer the reader to Warner's book \cite{Warner} for more details.  It is straightforward to verify that
\begin{equation}
G(x) + G(x/2) + ... + G(x/2^J) =1
\end{equation}for $x\in [1, 3\cdot 2^{J-1}]$ and is supported on $[3/4, 2^{J+1}]$.  We fix, once and for all, a function $G$ with the properties above.  

\section{Main Propositions}\label{sec:prop}
First, we let $\cL_0 \ge 100$ be a sufficiently large constant satisfying that the number of primes in the interval $[\sqrt{\cL}, \sqrt{2 \cL}]$ exceeds $\frac{\sqrt{\cL}}{2 \log \cL}$ for all $\cL \ge \cL_0$.  Recall that $\sumflat$ denotes a sum over fundamental discriminants.  For convenience, we further let
\begin{equation}
S(M, N, t) = \sum_{\substack{M \le |m| < 2M}} \left|\sum_{n} \frac{\lambda_f(n)}{n^{1/2 + it}} G\bfrac{n}{N} \bfrac{m}{n} \right|^2,
\end{equation}and
\begin{equation}
S^{\flat}(M, N, t) = \sumflat_{\substack{M \le |m| < 2M}} \left|\sum_{n} \frac{\lambda_f(n)}{n^{1/2 + it}} G\bfrac{n}{N} \bfrac{m}{n} \right|^2,
\end{equation}
for the fixed $G$ defined in \eqref{eqn:G}.  We record our inflation Lemma below.

\begin{lem}\label{lem:inflate}
Let $\cL_1 \ge \cL_0$ and $p$ be any odd prime.  With notation as above we have,
\begin{equation}
S(M, N, t) \ll S(p^2M, N, t) + \frac{1}{p} S(M, N/p, t) + \frac{1}{p^2} S(M, N/p^2, t),
\end{equation}and
\begin{align}
&S^{\flat}(M, N, t)  \\
&\ll \frac{\log \cL_1}{\sqrt{\cL_1}} \left(S(M\cL_1, N, t) + S(2M\cL_1, N, t) +   \sum_{\sqrt{\cL_1} \le p \le \sqrt{2\cL_1}} \left(\frac{S^{\flat}(M, N/p, t)}{p} + \frac {S^{\flat}(M, N/p^2, t)}{p^2}\right)\right) \notag ,
\end{align}
where the implied constants are absolute and in particular do not depend on $\cL_1$.
\end{lem}

Next, we state our key Proposition. 
\begin{prop}\label{prop:key}
For $M, N \ge 1$ and notation as above, there exists a constant $\cL \ge \cL_0$ depending only on $f$ such that
\begin{equation}
S^{\flat}(M, N, t) \leq   \cL^{2/3} (1+|t|)^2 (M + N \log (2+N/M) ).
\end{equation} 
\end{prop}

We have made no attempt to optimize the dependence on $t$ in Proposition \ref{prop:key}.  When $N$ is large, applying the functional equation gives a superior bound - see Lemma \ref{lem:indhyp2} for details.

\subsection{Notation}
We will be using an inductive argument to prove Proposition \ref{prop:key}, so it is important to ensure that our constant $\cL$ does not increase with each inductive step.  In what follows, we use the standard big-$O$ and Vinogradov notation with our implied constants never dependent on $\cL$.

\subsection{Outline}
Lemma \ref{lem:inflate} will be proven in \S \ref{sec:leminflate}.  The bulk of the work goes towards proving Proposition \ref{prop:key}, which is done in \S \ref{sec:propproof}.  Finally, the remaining details of the proof of Theorem \ref{thm:main} based on Proposition \ref{prop:key} is provided in \S \ref{sec:thmproof}.

\section{Proof of Lemma \ref{lem:inflate}} \label{sec:leminflate}

We write for any odd prime $p$,
\begin{align}\label{eqn:Sintro}
\left|\sum_{n} \frac{\lambda_f(n)}{n^{1/2 + iu}} G\bfrac{n}{N}  \bfrac{m}{n} \right|^2
&= \left|\sum_{\substack{n \\ p\nmid n}} \frac{\lambda_f(n)}{n^{1/2 + iu}}  \bfrac{m p^2}{n} G\bfrac{n}{N} + \sum_{\substack{ p|n}} \frac{\lambda_f(n)}{n^{1/2 + iu}}\bfrac{m}{n}  G\bfrac{n}{N}  \right|^2 \notag \\
&\le 2 \left|\sum_{n} \frac{\lambda_f(n)}{n^{1/2 + iu}} \bfrac{m p^2}{n} G\bfrac{n}{N} \right|^2 +2 \delta(p\nmid m)\left| \sum_{\substack{n}} \frac{\lambda_f(np)}{(np)^{1/2 + iu}}  \bfrac{m}{n} G\bfrac{np}{N} \right|^2, 
\end{align}
where $\delta(p\nmid m) = 1$ if $p \nmid m$ and vanishes otherwise.  We have also suppressed the condition $p \nmid n$ in the first sum, since $\bfrac{mp^2}{n} = 0$ otherwise.  By Hecke multiplicativity, $\lambda_f(np) = \lambda_f(n) \lambda_f(p) - \delta(p|n) \lambda_f(n/p)$ where $\delta(p|n) = 1$ when $p|n$ and vanishes otherwise.  Hence,
\begin{align}\label{eqn:S2}
\left| \sum_{\substack{n}} \frac{\lambda_f(np)}{(np)^{1/2 + iu}}  \bfrac{m}{n} G\bfrac{np}{N}  \right|^2
\le \frac{2|\lambda_f(p)|^2}{p} \left|\sum_{\substack{n}} \frac{\lambda_f(n) }{n^{1/2+iu}} \bfrac{m}{n} G\bfrac{np}{N} \right|^2 + \frac {2}{p^2}\left| \sum_{\substack{n}} \frac{\lambda_f(n)}{n^{1/2+iu}} \bfrac{m}{n} G\bfrac{np^2}{N} \right|^2.
\end{align}

By \eqref{eqn:Sintro} and \eqref{eqn:S2}, we conclude
\begin{equation}
S(M, N, u) \ll S(p^2M, N, u) + \frac{1}{p} S(M, N/p, u) + \frac{1}{p^2} S(M, N/p^2, u),
\end{equation}
which proves the first claim.

Further, by \eqref{eqn:Sintro} and \eqref{eqn:S2}, 

\begin{align}
\sum_{\sqrt{\cL_1} \le p \le \sqrt{2\cL_1}} S^{\flat}(M, N, u) 
&\ll \sum_{\sqrt{\cL_1} \le p \le \sqrt{2\cL_1}} \sumflat_{\substack{M \le |m| \le 2M }} \left|\sum_{n} \frac{\lambda_f(n)}{n^{1/2 + iu}} G\bfrac{n}{N} \bfrac{m p^2}{n}\right|^2 \notag \\
&+ \sum_{\sqrt{\cL_1} \le p \le \sqrt{2\cL_1}} \left(\frac{S^{\flat}(M, N/p, u)}{p} + \frac {S^{\flat}(M, N/p^2, u)}{p^2}\right) \notag \\
&\ll \sum_{\substack{M\cL_1 \le |m| < 4M \cL_1 }} \left|\sum_{n} \frac{\lambda_f(n)}{n^{1/2 + iu}} G\bfrac{n}{N} \bfrac{m}{n}\right|^2 \notag \\
&+ \sum_{\sqrt{\cL_1} \le p \le \sqrt{2\cL_1}} \left(\frac{S^{\flat}(M, N/p, u)}{p} + \frac {S^{\flat}(M, N/p^2, u)}{p^2}\right).
\end{align}
In the last line, we have used that when $m_1$ and $m_2$ are fundamental discriminants, $m_1 p_1^2 = m_2 p_2^2$ for odd primes $p_1, p_2$ only if $p_1 = p_2$.  This is because $m_i$ is either squarefree or is four times a squarefree number.  For $\cL_1 \ge \cL_0$, the number of primes in the interval $[\sqrt{\cL_1}, \sqrt{2\cL_1}]$ is $\ge \frac{\sqrt{\cL_1}}{2 \log \cL_1}$, and the second claim follows.

\section{Proof of Proposition \ref{prop:key}}\label{sec:propproof}

We proceed by induction on $M$.  The simple Lemma below will suffice for our base case.  
\begin{lem}\label{lem:bddnsum}
For $G$ the fixed function from \eqref{eqn:G}, $N>0$, and $m$ a fundamental discriminant, 
\begin{equation}
\left|\sum_{n} \frac{\lambda_f(n)}{n^{1/2+it}} G\bfrac{n}{N} \bfrac{m}{n} \right| \ll \sqrt{N_0} \log (N_0+2).
\end{equation}
where
\begin{equation}
N_0 = \min\left(N, \frac{|m|^2(1+|t|)^2}{N} \right) \le |m| (1+|t|),
\end{equation}and the implied constant is absolute.

\end{lem}
\begin{proof}
Note that 
\begin{equation}
\sum_{n} \left|\frac{\lambda_f(n)}{n^{1/2+it}} G\bfrac{n}{N} \bfrac{m}{n} \right| \ll \sum_{3N/4 \le n \le 2 N} \frac{d(n)}{\sqrt{n}} \\
\ll \sqrt{N} \log (N+2).
\end{equation}
Let $N_1 = \frac{|m|^2(1+|t|)^2}{N}$.  When $N \ge |m| (1+|t|)$, we apply \ref{lem:FE} to see that
\begin{align*}
\left|\sum_{n} \frac{\lambda_f(n)}{n^{1/2+it}} G\bfrac{n}{N} \bfrac{m}{n} \right|
&\le \left|\sum_{n} \frac{\lambda_f(n)}{n^{1/2-it}} \grave{G}_{it}\bfrac{4 \pi^2 nN}{|m|^2} \bfrac{m}{n} \right|\\
&\ll \sum_{n \le N_1} \frac{d(n)}{n^{1/2}} + \sum_{n > N_1} \frac{d(n)}{n^{1/2}} \bfrac{N_1}{n}^2,
\end{align*}
by \eqref{eqn:Fgravebdd} with $A = 2$.  The above is $\ll \sqrt{N_1} \log (N_1+2)$ which suffices.
\end{proof}

Lemma \ref{lem:bddnsum} implies that 
\begin{equation}\label{eqn:trivialbdd}
S^{\flat}(M, N, t) \ll M^2 (1+|t|) \log^2 (M(1+|t|)+2)  \ll M^2 (1+|t|)^2 \log^2(M+2) 
\end{equation}
where the implied constant $C'$ is absolute.  Thus the base case $M\le M_0$ is trivially true provided that $\cL^{2/3} \ge C' M_0 \log^2 (M_0+2)$.  For some fixed $M_1 \ge M_0$, our induction hypothesis is that for any $M\le M_1$ that
\begin{equation}\label{eqn:indhyp}
S^{\flat}(M, N, t) \le \cL^{2/3} (1+|t|)^2 (M + N \log (2+N/M) ), 
\end{equation}and we now proceed to prove \eqref{eqn:indhyp} for $M$ fixed with $M_1 < M \le M_1 + 1$. 

It will be convenient to proceed by a nested induction argument.  To be precise, we proceed by induction on $N$.  Note that the base case $N \le 2$ is trivial.   For clarity, let us note that our second induction hypothesis is that \eqref{eqn:indhyp} holds for our fixed $M$, any $|t|$ and all $N \le N_1$ for some $N_1 \ge 2$.  We now fix some $N$ with $N_1 < N \le N_1+1$.  We want to prove \eqref{eqn:indhyp} for our fixed $N$ and $M$.

We first record the following simple Lemma.  In what follows, we will use the inequalities in the Lemma without further explanation.

\begin{lem}\label{lem:inequalities}
For any complex numbers $a, b$ we have that
\begin{enumerate}
\item $(1+|a|)(1+|b|) \ge 1+|a|+|b|$.
\item If $|a|, |b| \gg 1$, then $|a||b| \gg |a|+|b|$.
\item $(1+|a|)(1+|b-a|) \ge 1+|b|$.
\end{enumerate}
\end{lem}

\begin{proof}
The first statement is clear, and the second statement follows from the first.  For the third, we note that $(1+|a|)(1+|b-a|) \ge 1+ |a| + |b-a| \ge 1+|b|$ by Triangle Inequality.
\end{proof}

The bound in \eqref{eqn:indhyp} becomes ineffectual when $N$ is very large compared to $M$.  The Lemma below rectifies that situation and will be the form of the induction hypothesis we most often use.

\begin{lem}\label{lem:indhyp2}
Suppose that \eqref{eqn:indhyp} holds for all $M \le M_1$ and all $N$ and $t$.  Then we also have that there exists some constant $C'$ depending only on $f$ such that
\begin{equation}\label{eqn:indhyp2}
S^{\flat}(M, N, t) \le C'\cL^{2/3} (1+|t|)^3 M \log (2+|t|) 
\end{equation}for all $M\le M_1$ and all $N$ and $t$.
\end{lem}
\begin{proof}
If $N \le M(1+|t|)$, then \eqref{eqn:indhyp2} follows immediately from \eqref{eqn:indhyp}.  Now suppose $N > M(1+|t|)$, and apply Lemma \ref{lem:FE} so that
\begin{align*}
S^{\flat}(M, N, t) = \sumflat_{\substack{M \le |m| < 2M }} \left|\sum_{n} \frac{\lambda_f(n)}{n^{1/2 - it}} \grave{G}_{it}\bfrac{4\pi^2 nN}{|m|^2} \bfrac{m}{n} \right|^2.
\end{align*}
We have that
\begin{align}\label{eqn:afterfe1}
\; \; \;&\sum_{n} \frac{\lambda_f(n)}{n^{1/2 - it}} \grave{G}_{it}\bfrac{4\pi^2 nN}{|m|^2} \bfrac{m}{n} \\
&=\frac{1}{2\pi i} \int_{(c)} \bfrac{|m|}{2\pi}^{2s} \frac{\Gamma\left(s -it +\frac{\kappa}{2}\right)}{\Gamma\left(-s + it + \frac{\kappa}{2}\right)}\sum_n\frac{\lambda_f(n)}{n^{1/2 - it+s}} \bfrac{m}{n} N^{-s} \tilde{G}(-s) ds \notag \\
&= \sumd_{N_0} \frac{1}{2\pi i} \int_{(c)} \bfrac{|m|}{2\pi}^{2s} \frac{\Gamma\left(s -it +\frac{\kappa}{2}\right)}{\Gamma\left(-s + it + \frac{\kappa}{2}\right)}\sum_n\frac{\lambda_f(n)}{n^{1/2 - it+s}} \bfrac{m}{n} G\bfrac{n}{N_0} N^{-s} \tilde{G}(-s) \notag ds,
\end{align}where $\sumd_{N_0}$ denotes a dyadic sum over $N_0 =  2^l$ for integer $l \ge 1$.  Now we let 
\begin{equation}\label{eqn:Vdef1}
V(x) = G(2x) + G(x) + G(x/2)
\end{equation}for the same fixed $G$ so that $V(x) = 1$ on $[3/4, 2]$.  Hence
\begin{align}\label{eqn:separaten}
&\sum_n\frac{\lambda_f(n)}{n^{1/2 - it+s}} \bfrac{m}{n} G\bfrac{n}{N_0}\\
&=\frac{1}{2\pi i} \int_{(\epsilon)} \sum_n\frac{\lambda_f(n)}{n^{1/2 - it+s}} \bfrac{m}{n} V\bfrac{n}{N_0} \bfrac{N_0}{n}^w \tilde{G}(w) dw \notag \\
&=\frac{1}{2\pi i} \int_{(\epsilon)} \sum_n\frac{\lambda_f(n)}{n^{1/2 - it+u}} \bfrac{m}{n} V\bfrac{n}{N_0} N_0^{u-s} \tilde{G}(u-s) du,\notag 
\end{align}  using a change of variables $u = s+w$.  

For convenience, let $N_1 = \frac{M^2(1+|t|)^2}{N}$.  When $N_0 \le N_1$, we move the contour of integration in $s$ to $\tRe s= -1$, and when $N_0 > N_1$, we move to $\tRe s = 4$.  In both cases, we also move the integration in $u$ to $\tRe u = 0$.  Using that $\tilde{G}(s) \ll \frac{1}{(1+|s|)^{10}}$, \eqref{eqn:Gammaestimate} and Lemma \ref{lem:inequalities}, we see that the quantity in \eqref{eqn:afterfe1} is bounded by $R(\le N_1) + R(> N_1)$ where
\begin{align*}
R(\le N_1) &= \sumd_{N_0 \le N_1} \int_{(-1)} \bfrac{N N_0}{M^2 (1+|t|)^2}  \int_{(0)} \left|\sum_n\frac{\lambda_f(n)}{n^{1/2 - it+u}} \bfrac{m}{n} V\bfrac{n}{N_0}\right| \frac{1}{|s|^{8} |u-s|^{10}} du ds \notag \\
&\ll \sumd_{N_0 \le N_1} \frac{N_0}{N_1} \int_{(0)} \left|\sum_n\frac{\lambda_f(n)}{n^{1/2 - it+u}} \bfrac{m}{n} V\bfrac{n}{N_0}\right| \frac{du}{(1+|u|)^6},
\end{align*}and similarly,
\begin{align*}
R(> N_1) \ll \sumd_{N_0 > N_1} \bfrac{N_1}{N_0}^4 \int_{(0)} \left|\sum_n\frac{\lambda_f(n)}{n^{1/2 - it+u}} \bfrac{m}{n} V\bfrac{n}{N_0}\right| \frac{du}{(1+|u|)^6}.
\end{align*}
By Cauchy-Schwarz, the definition of $V$ from \eqref{eqn:Vdef1} and the assumption of \eqref{eqn:indhyp}, the contribution of $R(\le N_1)$ to $S^{\flat}(M, N, t)$ is
\begin{align*}
&\ll \sumd_{N_0 \le N_1} \frac{N_0}{N_1} \int_{(0)} \sumflat_{\substack{M \le |m| < 2M}} \left|\sum_n\frac{\lambda_f(n)}{n^{1/2 - it+u}} \bfrac{m}{n} V\bfrac{n}{N_0}\right|^2 \frac{du}{(1+|u|)^6} \\
&\le \cL^{2/3} \sumd_{N_0 \le N_1} \frac{N_0}{N_1} \int_{(0)} (1+|u|+|t|)^2 (M + N_0 \log(2+N_0/M)) \frac{du}{(1+|u|)^6} \\
&\ll \cL^{2/3} (1+|t|)^2 \sumd_{N_0 \le N_1} \frac{N_0}{N_1} \int_{(0)}  (M + N_0 \log(2+N_0/M)) \frac{du}{(1+|u|)^4} \\
&\ll \cL^{2/3} (1+|t|)^2 (M + N_1 \log(2+N_1/M)),
\end{align*}where the implied constant is absolute.

Similarly, the contribution of $R(>N_1)$ to $S^{\flat}(M, N, t)$ is
\begin{align*}
&\ll \cL^{2/3} (1+|t|)^2 \sumd_{N_0 > N_1} \bfrac{N_1}{N_0}^4 \int_{(0)}  (M +  N_0 \log(2+N_0/M)) \frac{du}{(1+|u|)^4} \\
&\ll \cL^{2/3} (1+|t|)^2 (M + N_1 \log(2+N_1/M)),
\end{align*}where the implied constant is absolute.  Here we have used that $\frac{N_1}{N_0} \log(2+N_0/M)) \ll \log(2+N_1/M)$ for $N_1 > N_0$.  Since $N > M(1+|t|)$, $N_1 < M(1+|t|)$, and the Lemma follows.
\end{proof}

\begin{rem}
The proof of the Lemma above is involved because we have fixed the function $G$ inside the induction hypothesis \eqref{eqn:indhyp}, so it takes some effort to reduce the quantity in \eqref{eqn:afterfe1} to a suitable form.  This was done so that we do not need to keep track of constants which depend on our fixed $G$ in our future arguments.
\end{rem}

Now, let 
\begin{equation}
\cL_2 = \max\left(\cL  (1+|t|)^3 \bfrac{N}{M}^2, \cL (1+|t|)^3\right),
\end{equation}and
\begin{equation}
X = \cL_2 M,
\end{equation}  
where recall that $\cL$ is as in Proposition \ref{prop:key} and satisfies 
\begin{equation}
\cL \ge \cL_0 \ge 100
\end{equation}for convenience.  For clarity, we record the following useful bounds.

\begin{lem}\label{lem:sizecomputations}
With notation as above, we have that
\begin{equation}
\frac{\log \cL_2}{\sqrt{\cL_2}} X \ll (M + N \log(2+N/M))\cL^{3/5} (1+|t|)^{7/4},
\end{equation}where the implied constant is absolute.  Moreover, 
\begin{equation}
\frac{N^2}{X} \le \frac{M}{\cL (1+|t|)^3}.
\end{equation}
\end{lem}

\begin{proof}
Suppose $\cL_2 = \cL (1+|t|)^3$ so that $N \le M$, whence
\begin{equation}
\frac{N^2}{X} \le \frac{M}{\cL(1+|t|)^3}.  
\end{equation}
Also in this case, we have that $\frac{\log \cL_2}{\sqrt{\cL_2}} X = M \sqrt{\cL} (1+|t|)^{3/2} \log (\cL (1+|t|)^3) \ll \cL^{3/5} M (1+|t|)^{7/4}$.

In the complementary case when $\cL_2 = \cL  (1+|t|)^3 \bfrac{N}{M}^2$, we have that
\begin{equation}
\frac{N^2}{X} = \frac{N^2}{\cL_2 M} = \frac{M}{\cL (1+|t|)^3}.
\end{equation}
Also, 
\begin{align*}
\frac{\log \cL_2}{\sqrt{\cL_2}} X 
&= \sqrt{\cL}(1+|t|)^{3/2} N \log \cL_2\\
&\ll  \cL^{3/5} (1+|t|)^{7/4} N \log(2+N/M).
\end{align*}
\end{proof}

Applying Lemma \ref{lem:inflate} with $\cL_2$, we have that
\begin{align}\label{eqn:SMNT0}
&S^{\flat}(M, N, t)  \\
&\ll \frac{\log \cL_2}{\sqrt{\cL_2}} \left(S(X, N, t) + S(2X, N, t) +   \sum_{\sqrt{\cL_2} \le p \le \sqrt{2\cL_2}} \left(\frac{S^{\flat}(M, N/p, t)}{p} + \frac {S^{\flat}(M, N/p^2, t)}{p^2}\right)\right) \notag \\
&\ll \frac{\log \cL_2}{\sqrt{\cL_2}} \left(S(X, N, t) + S(2X, N, t) + \cL^{2/3} (M + N \log(2+N/M) )(1+|t|)^2 \right), \notag 
\end{align}where the implied constant depends only on $f$ and where we have applied the induction hypothesis in $N$ for the latter two terms.  We also have
\begin{align*}
S(X, N, t) + S(2X, N, t)
&= \sum_{\substack{X \le |m| < 4X}} \left|\sum_{n} \frac{\lambda_f(n)}{n^{1/2 + it}} G\bfrac{n}{N} \bfrac{m}{n} \right|^2\\
&\ll \sum_{k\ge 0} \frac{k+1}{2^{k/2}} \sum_{\substack{X \le |m| < 4X}}  \left|\sum_{(n, 2) =1} \frac{\lambda_f(n)}{n^{1/2 + it}} G\bfrac{n 2^k}{N} \bfrac{m}{n} \right|^2,
\end{align*}
by Hecke multiplicativity and Cauchy-Schwarz.  We now aim to show that

\begin{equation}\label{eqn:frakS}
\frak S := \sum_{\substack{X \le |m| < 4X}}  \left|\sum_{(n, 2) =1} \frac{\lambda_f(n)}{n^{1/2 + it}} G\bfrac{n}{N} \bfrac{m}{n} \right|^2 \ll X (1+|t|)^{1/4} + \cL^{2/3} N (1+|t|)^{3+2/5}.
\end{equation}

Before proceeding to the proof of \eqref{eqn:frakS}, we first verify that this suffices for Proposition \ref{prop:key}.  Indeed, applying \eqref{eqn:frakS} with $N/2^{k}$ in place of $N$ immediately implies that $S(X, N, t) + S(2X, N, t) \ll X (1+|t|)^{1/4} + \cL^{2/3} N (1+|t|)^{2/5}$.  Combining this with  \eqref{eqn:SMNT0}, we have that there exists some constants $\cal C_1, \cal C_2$ dependent only on $f$ such that
\begin{align*}
S^{\flat}(M, N, t) &\le \frac{\cal C_1 \log \cL_2}{\sqrt{\cL_2}} \left(X (1+|t|)^{1/4} + \cL^{2/3}N(1+|t|)^{3+2/5} + \cL^{2/3} (M + N \log(2+N/M) )(1+|t|)^2 \right)\\
&\le \cal C_2 \left(\cL^{3/5-2/3} + \cL^{-1/3}  \right) \cL^{2/3} (1+|t|)^2(M + N \log(2+N/M) )\\
&\le \cal \cL^{2/3} (1+|t|)^2(M + N \log(2+N/M) ),
\end{align*}
upon choosing $\cL$ sufficiently large compared to $\cal C_2$.  To derive the second line, we have used Lemma \ref{lem:sizecomputations} to bound $\frac{\cal \log \cL_2}{\sqrt{\cL_2}} X$ and that $\frac{\log \cL_2}{\sqrt{\cL_2}} \ll \frac{1}{\cL^{1/3} (1+|t|)^{1+2/5}}$ when $\cL_2 \ge \cL (1+|t|)^3$.

Now we proceed to prove \eqref{eqn:frakS}.  First, we fix $F$ to be a function satisfying the conditions in Lemma \ref{lem:testfunc} with 
\begin{equation}\label{eqn:Fcond}
c_0 = 1/16 \textup{  and  } c_1 = 4
\end{equation}  

Then by positivity,
\begin{align} \label{eqn:SXNT1}
\frak S
&\le \sum_{m} F\bfrac{m}{X} \left|\sum_{(n, 2)=1} \frac{\lambda_f(n)}{n^{1/2+it}} G\bfrac{n}{N} \bfrac{m}{n} \right|^2 \notag \\
&= \hat{F}(0)X \sumtwo_{\substack{n_1 n_2 = \square\\ (n_1 n_2, 2)=1}} \frac{\lambda_f(n_1) \lambda_f(n_2)}{n_1^{1/2+it} n_2^{1/2-it}} \prod_{p|n_1n_2} \left(1-\frac 1p\right)G\bfrac{n_1}{N}G\bfrac{n_2}{N} \\
&+ X \sum_{k \neq 0} \sumtwo_{\substack{n_1,n_2\\(n_1 n_2, 2)=1}} \frac{\lambda_f(n_1) \lambda_f(n_2)}{n_1^{1/2+it} n_2^{1/2-it}} \frac{G_k(n_1n_2)}{n_1 n_2} \check{F}\bfrac{kX}{n_1n_2}  G\bfrac{n_1}{N} G\bfrac{n_2}{N} \notag ,
\end{align}by Lemma \ref{lem:Poisson2}, and where

\begin{align} \label{eqn:checkFmellin1}
\check{F}\bfrac{kX}{n_1n_2} =  \frac{2}{(2\pi i)} \int_{(\epsilon)} \tilde{F}(1-s)\bfrac{n_1n_2}{2\pi X |k|}^s \Gamma(s) \cos \bfrac{\pi s}{2}  ds,
\end{align}since $F$ is even.  The following Lemma helps us understand the diagonal contribution. 
\begin{lem}\label{lem:calG}
Let
\begin{equation}\label{eqn:calG1def}
\cal{G}_1 (u, v) = \sumtwo_{\substack{n_1 n_2 = \square \\ (n_1 n_2, 2)=1 }} \frac{\lambda_f(n_1) \lambda_f(n_2)}{n_1^{1/2+u} n_2^{1/2+v}} \prod_{p|n_1n_2} \left(1-\frac 1p\right),
\end{equation}
and
\begin{equation}\label{eqn:calG2def}
\cal{G}_2 (u, v) = \sumtwo_{\substack{n_1 n_2 = \square \\ 2\nmid n_1n_2}} \frac{\lambda_f(n_1) \lambda_f(n_2)}{n_1^{1/2+u} n_2^{1/2+v}} \prod_{p|n_1n_2} \left(1-\frac{1}{p+1}\right).
\end{equation}

Then, for $i=1, 2$, we have
\begin{equation}
\cal{G}_i(u, v) = \zeta(1+u+v) L(1+2u, \sym^2 f) L(1+2v, \sym^2 f) L(1+u+v, \sym^2 f)  \cal{H}_i(u, v)
\end{equation}where $\cal{H}_i(u, v)$ converges absolutely in the region $\tRe u, v \ge -1/4 + \epsilon$.
\end{lem}
\begin{proof}
We prove the assertion for $\cal{G}_1$, the case for $\cal{G}_2$ being essentially the same.  We have that
\begin{align*}
\cal{G}_1 (u, v) = \prod_{p>2} \left(1 + \left(1-\frac 1p\right) \left(\sum_{k=1}^\infty \sum_{i+j = 2k} \frac{\lambda_f(p^i) \lambda_f(p^j)}{p^{i(1/2+u)} p^{j(1/2+u)}}\right) \right),
\end{align*}and the assertion follows upon noting that for  $\tRe u, v \ge -1/4 + \epsilon$,
\begin{align*}
\sum_{k\ge 1} \sum_{i+j = 2k} \frac{\lambda_f(p^i) \lambda_f(p^j)}{p^{i(1/2+u)} p^{j(1/2+u)}}
 = \frac{\lambda_f(p^2)}{p^{1+2u}} + \frac{\lambda_f(p^2)}{p^{1+2v}} + \frac{\lambda_f(p)^2}{p^{1+u+v}} + O\bfrac{1}{p^{1+4\epsilon}}.
\end{align*}

\end{proof}

By Mellin inversion and \eqref{eqn:calG1def}, the first term in \eqref{eqn:SXNT1} is
\begin{align*}
\hat{F}(0)X \frac{1}{(2\pi i)^2} \int_{(\epsilon)}\int_{(\epsilon)} 
\cal{G}_1(u+it, v-it) \tilde{G}(u) \tilde{G}(v) N^{u+v} du dv,
\end{align*}for any $1/16>\epsilon > 0$.  By Lemma \ref{lem:calG}, moving the contour in $u$ to $\tRe u = -2\epsilon$ picks up a pole at $u = -v$ and  gives that the above is
\begin{align}\label{eqn:diagbdd}
&\ll XN^{-\epsilon} (1+|t|)^{4\epsilon} + X \int_{(\epsilon)} |L(1-2v + 2it, \sym^2 f) L(1+2v - 2it, \sym^2 f)  \tilde{G}(-v) \tilde{G}(v)| dv \notag \\
&\ll X (1+|t|)^{1/4},
\end{align} 
where the implied constant depends only on $f$.  In the above, we have used the standard convexity estimate $L(\sigma + it, \sym^2 f) \ll (1+|t|)^{\frac{3}{2}(1-\sigma) + \epsilon}$ for $0\le \sigma \le 1$.  It remains to similarly bound the contribution of $k \neq 0$ in \eqref{eqn:SXNT1}, which we examine in the Proposition below.

\begin{prop}\label{prop:powersoftwobdd}
Let
\begin{equation}
T = \sum_{k\neq 0} \sumtwo_{\substack{n_1, n_2\\(n_1 n_2, 2)=1}} \frac{\lambda_f(n_1) \lambda_f(n_2)}{n_1^{1/2+it} n_2^{1/2-it}} \frac{G_k(n_1n_2)}{n_1 n_2} \check{F}\bfrac{kX}{n_1n_2}  G\bfrac{n_1}{N} G\bfrac{n_2}{N}.
\end{equation}

Then
$$T \ll \cL^{2/3} \frac{N}{X} (1+|t|)^{3+2/5},
$$where the implied constant may depend only on $f$.
\end{prop}
Note that Proposition \ref{prop:powersoftwobdd}, \eqref{eqn:SXNT1} and \eqref{eqn:diagbdd} immediately implies \eqref{eqn:frakS}.  We prove Proposition \ref{prop:powersoftwobdd} below.

\subsection{Proof of Proposition \ref{prop:powersoftwobdd}}\label{sec:proppowersoftwobddproof}
By \eqref{eqn:Fcond}, $\check{F}$ is supported on $[-1/16, 1/16]$ while $G$ is supported on $[3/4, 2]$ so that the range for $k$ in $T$ is restricted to
\begin{equation}
|k| \le \frac{n_1n_2}{16 X} \le  \frac{1}{4} \frac{N^2}{X}.  
\end{equation}

For convenience, we let 
\begin{equation}
K = \frac{N^2}{4X}.
\end{equation}

Further, we write $k = k_1 k_2^2$, where $k_1$ is squarefree, and $k_2$ is positive.  Using \eqref{eqn:checkFmellin1} to separate variables inside $\check{F}$ and the usual Mellin inversion on $G$ we have that
\begin{align*}
T &= \frac{2}{(2\pi i)^3} \int_{(\epsilon)} \int_{(2 \epsilon)} \int_{(2 \epsilon)} \bfrac{1}{2\pi X |k_1|}^s \tilde{F}(1-s)\Gamma(s) \cos \bfrac{\pi s}{2} \tilde{G}(u) \tilde{G}(v) \\
&\times \sumstar_{\substack{  |k_1| \le K}} Z(1/2+it+u-s, 1/2-it+v-s, s; k_1, 1) N^{u+v} du dv ds,
\end{align*}
 where recall from \eqref{eqn:Zdef} and Lemma \ref{lem:Z} that 
\begin{align*}
Z(\alpha, \beta, \gamma) = Z(\alpha, \beta, \gamma; k_1, q) 
&= \sum_{k_2 \ge 1}\sumtwo_{\substack{(n_1, 2q) = 1 \\ (n_2, 2q) = 1}} \frac{\lambda_f(n_1) \lambda_f(n_2)}{n_1^\alpha n_2^\beta k_2^{2 \gamma}} \frac{G_{k_1 k_2^2}(n_1 n_2)}{n_1 n_2}\\
&=  L(1/2+\alpha, f\otimes \chi_{m}) L(1/2+\beta, f\otimes \chi_{m})Y (\alpha, \beta,\gamma; k_1),
\end{align*}where $m=m(k_1)$ is the fundamental discriminant satisfying 
\begin{equation}
m = \begin{cases}
k_1 &\textup{ if } k_1 \equiv 1 \bmod 4\\
4k_1 &\textup{ if } k_1 \equiv 2, 3 \bmod 4.
\end{cases}
\end{equation}

Writing out the multiple Dirichlet series
\begin{equation}\label{eqn:Yseries}
Y (\alpha, \beta,\gamma; k_1) =  \sum_{r_1, r_2, r_3} \frac{C(r_1, r_2, r_2)}{r_1^\alpha r_2^\beta r_3^{2\gamma}},
\end{equation}
we also have 
\begin{align*}
Z(\alpha, \beta, \gamma) 
&= \sum_{r_1, r_2, r_3} \frac{C(r_1, r_2, r_2)}{r_1^\alpha r_2^\beta r_3^{2\gamma}} \sum_{n_1} \frac{\lambda_f(n_1) \chi_{m}(n_1)}{n_1^{1/2+\alpha}}\sum_{n_1} \frac{\lambda_f(n_2) \chi_{m}(n_2)}{n_2^{1/2+\beta}},
\end{align*}
so
\begin{align*}
&\frac{1}{(2\pi i)^2}  \int_{(\epsilon)} \int_{(\epsilon)} Z(1/2+it+u-s, 1/2-it+v-s, s) N^{u+v} \tilde{G}(u) \tilde{G}(v) du dv\\
&= \sum_{r_1, r_2, r_3} \frac{C(r_1, r_2, r_3)}{r_1^{1/2+it-s} r_2^{1/2-it-s} r_3^{2s}} \sum_{n_1} \frac{\lambda_f(n_1) \chi_{m}(n_1)}{n_1^{1+it-s}}\sum_{n_1} \frac{\lambda_f(n_2) \chi_{m}(n_2)}{n_2^{1-it-s}} G\bfrac{r_1n_1}{N} G\bfrac{r_2n_2}{N}\\
&= \sumd_{R_1, R_2} \sum_{r_1, r_2, r_3} \frac{C(r_1, r_2, r_3)}{r_1^{1/2+it-s} r_2^{1/2-it-s} r_3^{2s}}  G\bfrac{r_1}{R_2} G\bfrac{r_2}{R_2} \sum_{n_1} \frac{\lambda_f(n_1) \chi_{m}(n_1)}{n_1^{1+it-s}} \frac{\lambda_f(n_2) \chi_{m}(n_2)}{n_2^{1-it-s}} \\
&\times V\bfrac{n_1R_1}{N}V\bfrac{n_2R_2}{N} G\bfrac{r_1n_1}{N} G\bfrac{r_2n_2}{N},
\end{align*}
where we have applied a partition of unity to the sum over $r_1, r_2$.  To be more specific $\sumd_{R_1, R_2}$ denotes a sum over $R_i = 2^k$ for $k\ge 0$.  Since $G\bfrac{r_i n_i}{N}$ restricts $\frac{3}{4}\frac{N}{r_i} \le n_i \le \frac{2N}{r_i}$ while $G\bfrac{r_i}{R_i}$ restricts $\frac{3}{4}R_i \le r_i \le 2R_i$, the above holds for any $V$ which is identically $1$ on $[9/16, 4]$.  We set
\begin{equation}
V(x) = G(x/4) + G(x/2) + G(x) + G(2x),
\end{equation}where $G$ is our fixed function satisfying \eqref{eqn:Gdef} so that $V$ is identically $1$ on $[1/2, 6]$.

We once again apply Mellin inversion to separate variables inside $G\bfrac{r_1n_1}{N} G\bfrac{r_2n_2}{N}$, so that
\begin{align*}
T &= \sumd_{R_1, R_2} \frac{2}{(2\pi i)^3} \int_{(\epsilon)} \int_{(2 \epsilon)} \int_{(2 \epsilon)} \bfrac{1}{2 \pi X |k_1|}^s \tilde{F}(1-s)\Gamma(s) \cos \bfrac{\pi s}{2} \tilde{G}(u) \tilde{G}(v) \\
&\times \sumstar_{\substack{  |k_1| \le K}}  \sum_{r_1, r_2, r_3} \frac{C(r_1, r_2, r_3)}{r_1^{1/2+it-s+u} r_2^{1/2-it-s+v} r_3^{2s}}  G\bfrac{r_1}{R_2} G\bfrac{r_2}{R_2} \sum_{n_1} \frac{\lambda_f(n_1) \chi_{m}(n_1)}{n_1^{1+it-s+u}}  \\
&\times \frac{\lambda_f(n_2) \chi_{m}(n_2)}{n_2^{1-it-s+v}} V\bfrac{n_1R_1}{N}V\bfrac{n_2R_2}{N} N^{u+v} du dv ds.
\end{align*}

We now change variables $u-s$ to $u$ and $v+s$ to $v$.  Recalling the definition of $V$ as a finite sum of values of $G$, we have that $T$ is a finite sum of terms of the form
\begin{align}\label{eqn:Tdecomp1}
&\sumd_{R_1, R_2} \frac{2}{(2\pi i)^3} \int_{(3/5)} \int_{(-1/2)} \int_{(-1/2)} \bfrac{1}{\pi X}^s \tilde{F}(1-s)\Gamma(s) \cos \bfrac{\pi s}{2}   \notag \\
&\times \sumstar_{\substack{  |k_1| \le K}} \frac{1}{|k_1|^s} \sum_{r_1, r_2, r_3} \frac{C(r_1, r_2, r_3)}{r_1^{1/2+u+it} r_2^{1/2+v-it} r_3^{2s}} G\bfrac{r_1}{R_1} G\bfrac{r_2}{R_2} \sum_{n_1} \frac{\lambda_f(n_1) \chi_{k_1}(n_1)}{n_1^{1+u+it}}\sum_{n_1} \frac{\lambda_f(n_2) \chi_{k_1}(n_2)}{n_2^{1+v-it}} \notag \\
&\times G\bfrac{n_1}{N_1} G\bfrac{n_2}{N_2} N^{u+v+2s} \tilde{G}(u+s) \tilde{G}(v+s) du dv ds,
\end{align}where $N_i \asymp N/R_i$.  We have separated the variables $r_i$ from $n_i$.  We bound the sum over $r_i$ in the Lemma below.

\begin{lem}\label{lem:risumbdd}
Suppose $\tRe s \ge 3/5$ and write $u = -1/2+i\mu, v =-1/2+i\nu$ for real $\mu, \nu$.  We have that
\begin{equation}
\left|\sum_{r_1, r_2, r_3} \frac{C(r_1, r_2, r_3)}{r_1^{1/2+u+it} r_2^{1/2+v-it} r_3^{2s}} G\bfrac{r_1}{R_1} G\bfrac{r_2}{R_2}\right| \ll (1+|t|)^{1/5} (1+|\mu|)(1+|\nu|) \exp(-c_1 \sqrt{\log (R_1R_2)}).
\end{equation} 
where the implied constant and $c_1>0$ may depend only on $f$.
\end{lem}
\begin{proof}
We first quote the standard zero free region and lower bound for $L(s, \sym^2f)$.  There exists a positive constant $c$ depending only on $f$ such that for $s = \sigma+it$, there are no zeros \footnote{Note that since $f$ is fixed, the presence of a possible exceptional zero does not disturb us.} of $L(s, \sym^2f)$ in the region 
\begin{equation}\label{eqn:zerofree}
\sigma \ge 1-\frac{c}{\log(2+|t|)}. 
\end{equation}  We refer the reader to Theorem 5.42 of \cite{IK} for the proof.  Further, in the same region, we have the standard bound
\begin{equation}\label{eqn:lowerbddLfunc}
L(s, \sym^2f) \gg_f \frac{1}{\log(2+|t|)},
\end{equation}where the implied constant depends only on $f$.  For notational convenience, we assume that the same zero free region and bound holds true with $\zeta(s)$ in place of $L(s, \sym^2 f)$.

By Mellin inversion, 
\begin{align}\label{eqn:rsum1}
&\sum_{r_1, r_2, r_3} \frac{C(r_1, r_2, r_3)}{r_1^{1/2+u+it} r_2^{1/2+v-it} r_3^{2s}} G\bfrac{r_1}{R_1} G\bfrac{r_2}{R_2} \notag \\
&= \frac{1}{(2\pi i)^2} \int_{(0)} \int_{(0)} Y(i(\mu+ t)+\omega_1, i(\nu-t)+\omega_2, s) R_1^{\omega_1} R_2^{\omega_2} \tilde{G}(\omega_1) \tilde{G}(\omega_2) d\omega_1 d\omega_2.
\end{align}
By Lemma \ref{lem:Z} 
\begin{align*}
 &Y(i(\mu+ t)+\omega_1, i(\nu-t)+\omega_2, s)\\
 &=Z_2(i(\mu+ t)+\omega_1,i(\nu-t)+\omega_2 , s)\zeta(1+i(\mu+\nu) + \omega_1 + \omega_2)^{-1} \\
 &\times L(1+i(\mu+\nu) + \omega_1 + \omega_2, \sym^2 f)^{-1} L(1+2i(\mu + t) + 2 \omega_1, \sym^2 f)^{-1}  L(1+2i(\nu-t)+2\omega_2, \sym^2 f)^{-1},
\end{align*}
where $Z_2(i(\mu+ t)+\omega_1,i(\nu-t)+\omega_2 , s)$ is analytic and uniformly bounded in the region $\tRe \omega_1, \omega_2 \ge -1/20$ and $\tRe s \ge 3/5$, upon applying Lemma \ref{lem:Z} with $\delta = \frac {1}{10}$. 

We note for any non-negative numbers $a_1,...,a_n$ that 
\begin{equation}\label{eqn:logsumbdd}
\log(2+a_1+...+a_n) \ll \log(2+a_1)\times ... \times\log(2+a_n).
\end{equation}  Then within the zero free region where the bound \eqref{eqn:lowerbddLfunc} holds, we have that
\begin{align}\label{eqn:Ybdd}
&Y(i(\mu+ t)+\omega_1, i(\nu-t)+\omega_2, s) \notag \\
&\ll \log^2(2+|t|) \log^3(2+|\mu|) \log^3 (2+|\nu|)) \log^3(2+|\omega_1|) \log^3(2+|\omega_2|),
\end{align}where the implied constant depends only on $f$ and we may assume that $c$ is sufficiently small to force $\tRe \omega_1, \tRe \omega_2 \ge -\frac{1}{20}$ in the zero free region.

Let $T = \exp(\sqrt{\log R_1 R_2})$.  If $|\tIm \omega_i| > T$ for either $i=1$ or $i=2$, the bound $\tilde{G}(\omega_i) \ll_A (1+|\omega_i|)^{-A}$ for any $A>0$ implies that this range gives a contribution of $\ll  \log^2(2+|t|) \log^3(2+|\mu|) \log^3 (2+|\nu)) \exp(-\sqrt{\log R_1 R_2})$ to \eqref{eqn:rsum1}.

Now, suppose that $|\tIm \omega_i| \le T$ for $i=1$ and $i=2$.  If $\max(|t|, |\mu|, |\nu|) \le T$, then we move the contour of integration to $\tRe \omega_i = -\frac{c'}{\log T} \ge -\frac{1}{20}$ for some $1 \ge c'>0$ depending only on $f$ such that \eqref{eqn:Ybdd} is satisfied.  Such $c'$ exists by a quick examination of \eqref{eqn:zerofree}, \eqref{eqn:lowerbddLfunc} and 
\eqref{eqn:logsumbdd}.  The contribution of this is $\ll \log^2(2+|t|) \log^3(2+|\mu|) \log^3 (2+|\nu)) \exp(-c'/2 \sqrt{\log R_1 R_2})$ using the bound $\tilde{G}(\omega_i) \ll_A (1+|\omega_i|)^{-A}$ to bound the short horizontal contribution.  

Now if $\max(|t|, |\mu|, |\nu|) > T$, we do not shift contours and obtain that \eqref{eqn:rsum1} is
\begin{align*}
&\ll \log^2(2+|t|) \log^3(2+|\mu|) \log^3 (2+|\nu))\\
& \ll (1+|t|)^{1/5}(1+|\mu|)(1+|\nu|) \exp(-1/6 \sqrt{\log (R_1R_2)}).
\end{align*}
\end{proof}

Using that $\Gamma(s) \cos \bfrac{\pi s}{2} \ll |s|^{\tRe s - 1/2}$ and by Lemma \ref{lem:risumbdd} we have that the quantity in \eqref{eqn:Tdecomp1} is
\begin{align}\label{eqn:Tdecomp2}
&\ll (X)^{-3/5} N^{1/5} (1+|t|)^{1/5}  \sumd_{R_1, R_2} \exp(-c_1 \sqrt{\log (R_1R_2)}) \int_{-\infty}^{\infty} \int_{\infty}^{\infty}  \bfrac{1}{(1+|\mu|)(1+|\nu|)}^{10} \notag \\
&\times \sumstar_{\substack{  |k_1| \le K}} \frac{1}{|k_1|^{3/5}} \left| \sum_{n_1} \frac{\lambda_f(n_1) \chi_{m}(n_1)}{n_1^{1/2+i\mu+it}}\sum_{n_1} \frac{\lambda_f(n_2) \chi_{m}(n_2)}{n_2^{1/2+i\nu-it}} G\bfrac{n_1}{N_1} G\bfrac{n_2}{N_2} \right| d\mu d\nu.
\end{align}

We split the sum in $k_1$ into dyadic intervals of the form $\cal K_1\le |k_1| < 2\cal K_1$, and let
\begin{align*}
\cal{T}(\cal K_1) = \cal{T}(\cal K_1; t, x, \cal N) 
&= \sumstar_{\substack{  \cal K_1\le |k_1| < 2\cal K_1}}\left| \sum_{n} \frac{\lambda_f(n) \chi_{m}(n)}{n^{1/2+ix+it}}G\bfrac{n}{\cal N} \right|^2\\
&\leq \sumflat_{\cal K_1 \le |m| < 8\cal K_1} \left| \sum_{n} \frac{\lambda_f(n) \chi_{m}(n)}{n^{1/2+ix+it}}G\bfrac{n}{\cal N} \right|^2\\
&\ll \cL^{2/3} (1+|t+x|)^3 \cal{K}_1 \log (2+|t+x|) )\\
&\ll \cL^{2/3} (1+|t|)^{3+1/5} (1+|x|)^{4} \cal{K}_1 
\end{align*}where the implied constant is absolute, and where we may apply the induction hypothesis in the form given by \eqref{eqn:indhyp2} since $4 \cal K_1 \le 4K = \frac{N^2}{X} \le \frac{M}{\cL} < M-1$, where the second last estimate is from Lemma \ref{lem:sizecomputations}.

Putting in this estimate, we see by Cauchy-Schwarz and dyadic summation over $\cal K_1$ that
\begin{align*}
&\sumstar_{\substack{  |k_1| \le K}} \frac{1}{|k_1|^{3/5}} \left| \sum_{n_1} \frac{\lambda_f(n_1) \chi_{m}(n_1)}{n_1^{1/2+i\mu+it}}\sum_{n_1} \frac{\lambda_f(n_2) \chi_{m}(n_2)}{n_2^{1/2+i\nu-it}} G\bfrac{n_1}{N_1} G\bfrac{n_2}{N_2} \right|\\
&\ll \cL^{2/3} (1+|t|)^{3+1/5} (1+|\mu|)^{2} (1+|\nu|)^{2} K^{2/5}.  
\end{align*}

Putting this into \eqref{eqn:Tdecomp2} and recalling that $K = \frac{N^2}{4X}$ gives a bound of
\begin{align*}
\ll \cL^{2/3} \frac{N}{X} (1+|t|)^{3+2/5},
\end{align*}where we have estimated $\sumd_{R_1, R_2}\exp(-c_1 \sqrt{\log (R_1R_2)}) = \sum_{l, k \ge 0} \exp(-c_1 \sqrt{\log 2} \sqrt{l+k}) = \sum_{h \ge 0} (h+1) \exp(-c_1 \sqrt{\log 2} \sqrt{h}) \ll 1$.  This completes the proof of Proposition \ref{prop:powersoftwobdd}.

\section{Proof of Theorem \ref{thm:main}}\label{sec:thmproof}
It suffices to show that 
\begin{equation}\label{eqn:smoothedmoment}
\sumstar_{(d, 2)=1} L(1/2, f\otimes \chi_{8d})^2 J\bfrac{8d}{X} = \frac{2X}{\pi^2} \tilde{J}(1) L(1, \sym^2 f)^3 \cal H_2(0, 0) \log X + O(X (\log X)^{1/2+\epsilon})
\end{equation}
for any smooth nonnegative function $J$ compactly supported on $[1/2, 2]$ and where $\cal H_2$ is defined as in Lemma \ref{lem:calG}.  Recall from Lemma \ref{lem:approximatefunctionalequation} that for $d$ odd and squarefree,
\begin{equation}
L(1/2, f\otimes\chi_{8d}) = 2 \sum_{n\ge 1} \frac{\lambda_f(n) \chi_{8d}(n)}{n^{1/2}} W_{1/2}\bfrac{n}{8|d|},
\end{equation}
where for any $c>0$,
\begin{equation}
W_{1/2}(x) = \frac{1}{2\pi i} \int_{(c)} \frac{\Gamma(\frac{\kappa}{2}+w)}{\Gamma(\frac{\kappa}{2})} (2\pi x)^{-w} \frac{dw}{w}.
\end{equation}
We let $\cal N = \frac{X}{(\log X)^{100}}$ and further let
\begin{equation}\label{eqn:calA}
\cal A(8d) = \cal A(8d, \cal N) = 2 \sum_{n\ge 1} \frac{\lambda_f(n) \chi_{8d}(n)}{n^{1/2}} W_{1/2}\bfrac{n}{\cal N}
\end{equation}and
\begin{equation}
\cal B(8d) = L(1/2, f\otimes\chi_{8d}) - \cal A(8d).
\end{equation}

Note that since $W_{1/2}(x)$ is real for all $x$, $\cal A(8d)$ and $\cal B(8d)$ are real as well.  Then \eqref{eqn:smoothedmoment} follows from the Propositions below.

\begin{prop} \label{prop:calBbdd}
With notation as above,
\begin{equation}
\sumstar_{(d, 2)=1} \cal B(8d)^2  J\bfrac{8d}{X} \ll X (\log \log X)^4.
\end{equation}
\end{prop}

\begin{prop}\label{prop:calAcalc}
With notation as above,
\begin{equation}
\sumstar_{(d, 2)=1} \cal A(8d)^2  J\bfrac{8d}{X} = \frac{2X}{\pi^2} \tilde{J}(1) L(1, \sym^2 f)^3 \cal H_2(0, 0) \log X + O(X \log \log X).
\end{equation}
\end{prop}

Indeed, assuming Propositions \ref{prop:calBbdd} and \ref{prop:calAcalc}, we have by Cauchy-Schwarz that
\begin{equation}\label{eqn:CSbdd}
\sumstar_{(d, 2)=1} \cal A(8d) \cal B(8d) J\bfrac{8d}{X} \ll (X\log X)^{1/2} X^{1/2} (\log \log X)^2 \ll X (\log X)^{1/2+\epsilon},
\end{equation}and \eqref{eqn:smoothedmoment} follows.  With more care, one can improve the bound in \eqref{eqn:CSbdd} by a direct evaluation similar to the proof of Proposition \ref{prop:calAcalc}, hence improving the error term in \eqref{eqn:smoothedmoment} to $O(X (\log \log X)^4)$.  Actually, with even more effort, one should be able to slightly refine the bound in Proposition \ref{prop:calBbdd} with a more general form of Proposition \ref{prop:key}.  We now proceed the proofs of Propositions \ref{prop:calBbdd} and \ref{prop:calAcalc}.

\subsection{Proof of Proposition \ref{prop:calBbdd}}\label{sec:proofpropcalBbdd}
For our fixed $G$ with proporties as in \eqref{eqn:Gdef}, we have that
\begin{align*}
\cal B(8d) = \sumd_N \frac{1}{\pi i}  \int_{(1)} \frac{\Gamma(\frac{\kappa}{2}+w)}{(2\pi)^w \Gamma(\frac{\kappa}{2})} \frac{(8|d|)^{w} - \cal N^{w}}{w} \sum_{n\ge 1} \frac{\lambda_f(n) \chi_{8d}(n)}{n^{1/2+w}} G\bfrac{n}{N} dw,
\end{align*}where $\sumd_N$ denotes a sum over $N =2^k$ for $k \ge 0$.  Now we proceed in a manner similar to Lemma \ref{lem:indhyp2}.  Indeed, by \eqref{eqn:separaten},
\begin{align}
&\sum_n\frac{\lambda_f(n)}{n^{1/2 +w}} \bfrac{8m}{n} G\bfrac{n}{N}\\
&=\frac{1}{2\pi i} \int_{(\epsilon)} \sum_n\frac{\lambda_f(n)}{n^{1/2 +u}} \bfrac{8m}{n} V\bfrac{n}{N} N^{u-w} \tilde{G}(u-w) du,\notag 
\end{align}with $V$ as in \eqref{eqn:Vdef1}.  

When $N \le \mathcal N$, we move the contour in $w$ to $\tRe w = -1$, noting that we pass no poles in the process.  For $\mathcal N < N \le X$, we move the integral in $w$ to $\tRe w = 0$, while for $N > X$, we move to $\tRe w = 4$.  In all cases, we move the contour in $u$ to $\tRe u = 0$.  By Cauchy-Schwarz, the contribution of those $N\le \mathcal N$ is
\begin{align*}
\sumd_{N \le \mathcal {N}} \frac{N}{\mathcal N} \int_{(-1)} \int_{-\infty}^\infty \sumstar_{(d, 2) = 1} \left|\sum_n\frac{\lambda_f(n)}{n^{1/2 +it}} \bfrac{8d}{n} V\bfrac{n}{N}\right|^2 J\bfrac{8d}{X} \frac{dt}{(1+|it-w|)^{10}} \frac{dw}{(1+|w|)^{10}}.
\end{align*}Note that $8d$ is a fundamental discriminant when $(d, 2) = 1$ and $d$ is squarefree.  Moreover, $V$ is as in \eqref{eqn:Vdef1}, so we may apply Proposition \ref{prop:key} to see that the above is
\begin{align*}
\ll \sumd_{N \le \mathcal {N}} \frac{N}{\mathcal N} (X+ N\log(2+N/X)) 
\ll X.
\end{align*}

In the range $\mathcal N < N \le |8d|$, we note that for $\tRe w = 0$, $\frac{(8|d|)^{w} - \cal N^{w}}{w} \ll \log \bfrac{8|d|}{\mathcal N} \ll \log \frac{X}{\cal N}$.  Moreover, $\sumd_{\mathcal {N} < N \le 8|d|} 1 \ll \log \frac{X}{\cal N}$.  Thus, similar to the above, this range contributes
\begin{align*}
&\log^3 \frac{X}{\cal N} \sumd_{\mathcal {N} < N \le 8|d|} \int_{(0)} \int_{-\infty}^\infty \sumstar_{(d, 2) = 1} \left|\sum_n\frac{\lambda_f(n)}{n^{1/2 +it}} \bfrac{8m}{n} V\bfrac{n}{N}\right|^2 J\bfrac{8d}{X} \frac{dt}{(1+|it-w|)^{10}} \frac{dw}{(1+|w|)^{10}}\\
&\ll X (\log \log X)^4,
\end{align*}again using Proposition \ref{prop:key}.  Finally, the range $N > X$ contributes
\begin{align*}
&\sumd_{N > X} \bfrac{X}{N}^4 \int_{(4)} \int_{-\infty}^\infty \sumstar_{(d, 2) = 1} \left|\sum_n\frac{\lambda_f(n)}{n^{1/2 +it}} \bfrac{8d}{n} V\bfrac{n}{N}\right|^2 J\bfrac{8d}{X} \frac{dt}{(1+|it-w|)^{10}} \frac{dw}{(1+|w|)^{10}} \\
&\ll \sumd_{N > X} \bfrac{X}{N}^4 N \log (2+N/X) \\
&\ll X.
\end{align*}by Proposition \ref{prop:key}.  In the last time, we have used that $\frac{X}{N} \log (2+N/X) \ll 1$ for $N>X$.

\subsection{Proof of Proposition \ref{prop:calAcalc}}
We have that 
\begin{align*}
\sumstar_{(d, 2)=1} \cal A(8d)^2  J\bfrac{8d}{X} =\sum_{(d, 2)=1} \cal A(8d)^2  J\bfrac{8d}{X} \sum_{a^2|d} \mu(a).
\end{align*}
Exchanging the order of summation, we see that
\begin{align*}
&\sumstar_{(d, 2)=1} \cal A(8d)^2  J\bfrac{8d}{X} \\
&= 4\left(\sum_{\substack{a \le Y \\ (a, 2) = 1}}+\sum_{\substack{a>Y \\ (a, 2) = 1}}\right) \mu(a) \sum_{(d, 2) = 1}J\bfrac{8da^2}{X} \left|\sum_{(n, a) = 1 } \frac{\lambda_f(n) }{n^{1/2}} \chi_{8d}(n) W_{1/2} \bfrac{n}{\cal N} \right|^2,
\end{align*}where for concreteness, we set $Y = \log^{20} X$.

\subsubsection{The contribution of $a > Y$}
We first prove the following Lemma.

\begin{lem}\label{lem:coprimeprimitive}
For any real $\cal X, N \ge 1$, real $t$, and positive integer $q$
\begin{equation}\label{eqn:coprimenonprimitivesum}
\sum_{\substack{(d, 2) = 1 \\ d\le \cal X}} \left| \sum_{\substack{(n, q) = 1}} \frac{\lambda_f(n) }{n^{1/2+it}} \bfrac{8d}{n} G\bfrac{n}{N}\right|^2 \ll d(q)^5  \cal X(1+|t|)^3 \log(2+|t|)
\end{equation}
\end{lem}

\begin{proof}
We first write $d = b^2 m$ where $m$ is squarefree, so that the left side of \eqref{eqn:coprimenonprimitivesum} is
\begin{equation}\label{eqn:extractprimitive}
\sum_{b\ge 1} \sumstar_{\substack{(m, 2) = 1 \\ m\le \cal X/b^2}} \left| \sum_{\substack{(n, bq) = 1}} \frac{\lambda_f(n) }{n^{1/2+it}} \bfrac{8m}{n} G\bfrac{n}{N}\right|^2.
\end{equation}

We further write 
\begin{align*}
& \left|\sum_{\substack{(n, bq) = 1}} \frac{\lambda_f(n) }{n^{1/2+it}} \bfrac{8m}{n} G\bfrac{n}{N}\right|\\
&  = \left|\sum_{c|bq}\frac{\mu(c)}{c^{1/2+it}} \bfrac{8m}{c} \sum_{n} \frac{\lambda_f(nc) }{n^{1/2+it}} \bfrac{8m}{n} G\bfrac{nc}{N} \right|\\
&  \le \sum_{c|bq} \sum_{e|c} \frac{d(c/e)}{\sqrt{ce}} \left|\sum_{n} \frac{\lambda_f(n) }{n^{1/2+it}} \bfrac{8m}{n} G\bfrac{nce}{N}\right|.
\end{align*}
Combining this with \eqref{eqn:extractprimitive}, applying Cauchy-Schwarz, and crudely bounding the divisor sums, we see that the left side of \eqref{eqn:coprimenonprimitivesum} is bounded by
\begin{align*}
&\sum_{b\ge 1} d(bq)^3 \sum_{c|bq} \sum_{e|c}   \sumstar_{\substack{(m, 2) = 1 \\ m\le \cal X/b^2}} \left|\sum_{n} \frac{\lambda_f(n) }{n^{1/2+it}} \bfrac{8m}{n} G\bfrac{nce}{N}\right|^2 \\
&\ll \sum_{b\ge 1} d(bq)^3 \sum_{c|bq} \sum_{e|c}   \frac{\cal X}{b^2} (1+|t|)^3 \log(2+|t|) \\
&\ll d(q)^5  \cal X(1+|t|)^3 \log(2+|t|),
\end{align*}as claimed, where we have used the bound \eqref{eqn:indhyp2} and dyadic summation to derive the second line.  
\end{proof}

Similar to the proof of Proposition \ref{prop:calBbdd} in \S \ref{sec:proofpropcalBbdd},
\begin{align*}
&\sum_{(n, a) = 1} \frac{\lambda_f(n)}{n^{1/2}} \chi_{8d}(n) W_{1/2} \bfrac{n}{\cal N}\\
&\ll  \sumd_N \int_{(1)} \frac{\Gamma(\frac{\kappa}{2}+w)}{(2\pi)^w \Gamma(\frac{\kappa}{2})} \frac{\cal N^{w}}{w} \frac{1}{2\pi i} \int_{(\epsilon)} \sum_{(n, a)=1} \frac{\lambda_f(n)}{n^{1/2 +u}} \bfrac{8m}{n} V\bfrac{n}{N} N^{u-w} \tilde{G}(u-w) du dw,
\end{align*}where $\sumd_N$ denotes a sum over $N =2^k$ for $k \ge 0$.  Applying Cauchy-Schwarz, Lemma \ref{lem:coprimeprimitive} and some calculation similar to that in \S \ref{sec:proofpropcalBbdd} then yields that 
\begin{align*}
&\sum_{(d, 2) = 1}J\bfrac{8da^2}{X} \left|\sum_{(n, a) = 1} \frac{\lambda_f(n)}{n^{1/2}} \chi_{8d}(n) W_{1/2} \bfrac{n}{\cal N}\right|^2 \\
&\ll \frac{X}{a^2} d(a)^5.
\end{align*}
Hence,
\begin{align*}
\sum_{\substack{a>Y \\ (a, 2) = 1}} \sum_{(d, 2) = 1}J\bfrac{8da^2}{X}  \left|\sum_{(n, a) = 1 } \frac{\lambda_f(n) }{n^{1/2}} \chi_{8d}(n) W_{1/2} \bfrac{n}{\cal N} \right|^2
\ll_\epsilon \frac{X}{Y^{1-\epsilon}},
\end{align*}for any $\epsilon > 0$.  Recalling that $Y = \log^{20} X$, this is absorbed into the error term in Proposition \ref{prop:calAcalc}.

\subsubsection{The contribution of $a\le Y$}
The computation is very similar to parts of the proof of Proposition \ref{prop:key} in \S \ref{sec:propproof}.  We have
\begin{align}\label{eqn:aleY}
&\sum_{(d, 2) = 1}J\bfrac{8da^2}{X} \left|\sum_{(n, a) = 1 } \frac{\lambda_f(n) }{n^{1/2}} \chi_{8d}(n) W_{1/2} \bfrac{n}{\cal N} \right|^2 \notag \\
&= \sum_{(d, 2) = 1}J\bfrac{8da^2}{X} \sumtwo_{(n_1 n_2, a) = 1 } \frac{\lambda_f(n_1)\lambda_f(n_2) }{(n_1 n_2)^{1/2}} \chi_{8d}(n_1n_2) W_{1/2} \bfrac{n_1}{\cal N} W_{1/2} \bfrac{n_2}{\cal N} \notag \\
& = \frac{\check{J}(0)X}{16 a^2} \sumtwo_{\substack{n_1 n_2 = \square \\ (2a , n_1n_2) = 1}} \frac{\lambda_f(n_1) \lambda_f(n_2)}{n_1^{1/2} n_2^{1/2}} \prod_{p|n_1n_2} \left(1-\frac 1p\right)W_{1/2}\bfrac{n_1}{\cal  N}W_{1/2} \bfrac{n_2}{\cal N} \\
&+ \frac{X}{16 a^2} \sum_{k \neq 0} (-1)^k \sumtwo_{(2a , n_1n_2) = 1} \frac{\lambda_f(n_1) \lambda_f(n_2)}{n_1^{1/2} n_2^{1/2}} \frac{G_k(n_1n_2)}{n_1 n_2} \check{J}\bfrac{kX}{16 a^2  n_1n_2}  W_{1/2}\bfrac{n_1}{\cal N} W_{1/2}\bfrac{n_2}{\cal N} \notag
\end{align}by Lemma \ref{lem:Poisson2}.  We record the contribution of the first term in the Lemma below.
\begin{lem}\label{lem:diagterm}
For $Y \ge \log^{10} X$,
\begin{align}\label{eqn:diagterm}
&4\sum_{\substack{(a, 2)=1 \\ a\le Y}} \frac{\hat{J}(0)X  \mu(a)}{16 a^2} \sumtwo_{\substack{n_1 n_2 = \square \\ (2a , n_1n_2) = 1}} \frac{\lambda_f(n_1) \lambda_f(n_2)}{n_1^{1/2} n_2^{1/2}} \prod_{p|n_1n_2} \left(1-\frac 1p\right)W_{1/2}\bfrac{n_1}{\cal  N}W_{1/2} \bfrac{n_2}{\cal N} \\
&= \frac{2\tilde{J}(1) X \log \cal N}{\pi^2} L(1, \sym^2 f)^3 \cal H_2(0, 0) + O(X \log \log X) \notag
\end{align}
\end{lem}
\begin{proof}
Note that
\begin{equation}
\sum_{\substack{a\le Y \\ (a, 2n_1n_2) = 1}} \frac{\mu(a)}{a^2} = \frac{1}{\zeta(2)} \prod_{p|2n_1n_2} \left(1-\frac{1}{p^2} \right)^{-1} + O(1/Y) = \frac{8}{\pi^2} \prod_{p|n_1n_2}\left(1-\frac{1}{p^2} \right)^{-1} + O(1/Y).
\end{equation}
Switching the order of summation gives that the left side of \eqref{eqn:diagterm} is
\begin{align*}
&\frac{2X}{\pi^2} \tilde{J}(1)  \sumtwo_{\substack{n_1 n_2 = \square \\ (2, n_1n_2) = 1}} \frac{\lambda_f(n_1) \lambda_f(n_2)}{n_1^{1/2} n_2^{1/2}} \prod_{p|n_1n_2} \left(1- \frac {1}{p+1}\right)W_{1/2}\bfrac{n_1}{\cal  N}W_{1/2} \bfrac{n_2}{\cal N} \\
&+ O\left(\frac{X}{Y}\sumtwo_{\substack{n_1 n_2 = \square}} \frac{d(n_1) d(n_2)}{n_1^{1/2} n_2^{1/2}} \left| W_{1/2}\bfrac{n_1}{\cal  N}W_{1/2} \bfrac{n_2}{\cal N}\right|   \right).
\end{align*}
The error term inside the big-$O$ is $\ll \frac{X}{Y} \log^{10} X$ by standard methods.  The main term is
\begin{align*}
&\frac{2X}{\pi^2} \tilde{J}(1)   \bfrac{1}{2\pi i}^2 \int_{(\epsilon)}\int_{(\epsilon)} \frac{\Gamma(\frac{\kappa}{2}+w_1)}{\Gamma(\frac{\kappa}{2})}\frac{\Gamma(\frac{\kappa}{2}+w_2)}{\Gamma(\frac{\kappa}{2})} \cal G_2(w_1, w_2) \bfrac{\cal N}{2\pi}^{w_1+w_2} \frac{dw_1}{w_1}\frac{dw_2}{w_2}
\end{align*}where by Lemma \ref{lem:calG} 
$$\cal{G}_2(w_1, w_2) = \zeta(1+w_1+w_2) L(1+2w_1, \sym^2 f) L(1+2w_2, \sym^2 f) L(1+w_1+w_2, \sym^2 f)  \cal{H}_2(w_1, w_2)
$$where $\cal{H}_2(w_1, w_2)$ converges absolutely in the region $\tRe w_1, w_2 \ge -1/4 + \epsilon$.

We now shift the contour of integration in $w_1$ to $\tRe w_1 = -1/5$.  We encounter two poles along the way: $w_1 = -w_2$ and $w_1 = 0$.  The contribution of the remaining integral at $\tRe w_1 = -1/5$ contributes $\ll X N^{-1/5+\epsilon}$.  The contribution of $w_1 = -w_2$ is
\begin{align*}
&\ll X \int_{(\epsilon)} \frac{\Gamma(\frac{\kappa}{2}-w_2)}{\Gamma(\frac{\kappa}{2})}\frac{\Gamma(\frac{\kappa}{2}+w_2)}{\Gamma(\frac{\kappa}{2})}L(1-2w_2, \sym^2 f) L(1+2w_2, \sym^2 f) L(1, \sym^2 f)  \cal{H}_2(-w_2, w_2) \frac{dw_2}{w_2^2}\\
&\ll X.
\end{align*}
Finally, the contribution of $w_1 = 0$ gives
\begin{align*}
&\frac{2\tilde{J}(1) X}{\pi^2}   \frac{L(1, \sym^2 f)}{2\pi i} \int_{(\epsilon)} \frac{\Gamma(\frac{\kappa}{2}+w_2)}{\Gamma(\frac{\kappa}{2})} \\
&\times \zeta(1+w_2)  L(1+2w_2, \sym^2 f) L(1+w_2, \sym^2 f)  \cal{H}_2(0, w_2) \bfrac{\cal N}{2\pi}^{w_2} \frac{dw_2}{w_2}.
\end{align*}
We similarly shift to $\tRe w_2 =  -1/5$, the contour there giving a contribution of $\ll XN^{-1/5+\epsilon}$, while the double pole at $w_2 = 0$ gives
$$\frac{2\tilde{J}(1) X \log \cal N}{\pi^2} L(1, \sym^2 f)^3 \cal H_2(0, 0) + O(X), 
$$which suffices upon noting that $\log \cal N = \log X + O(\log \log X)$.
\end{proof}

We now study the contribution of the $k\neq 0$ terms in \eqref{eqn:aleY}.  To be precise, we will prove that
\begin{align}\label{eqn:kneq0}
\sum_{k \neq 0} (-1)^k \sumtwo_{(2a , n_1n_2) = 1} \frac{\lambda_f(n_1) \lambda_f(n_2)}{n_1^{1/2} n_2^{1/2}} \frac{G_k(n_1n_2)}{n_1 n_2} \check{J}\bfrac{kX}{2 a^2 n_1n_2}  W_{1/2}\bfrac{n_1}{\cal N} W_{1/2}\bfrac{n_2}{\cal N} \ll \frac{a^{2+\epsilon} \cal N}{X}.
\end{align}
This gives a total contribution of $\sum_{a\le Y}\frac{a^{\epsilon} \cal N}{X} \ll Y^{1+\epsilon} \cal N \ll \frac{X}{\log^{50} X}$, which can be absorbed into the error term in Proposition \ref{prop:calAcalc}.

The proof of \eqref{eqn:kneq0} is very similar to the proof of Proposition \ref{prop:powersoftwobdd} and we provide a sketch here.  The presence of the factor $(-1)^k$ causes some minor issues, which we avoid by writing the quantity in \eqref{eqn:kneq0} as $T_0 - \sum_{l \ge 1} T_{l}$ where $T_{l}$ is the contribution of those $k$ with $2^{l} \| k$.  For odd $n$, $G_{4k}(n) = G_k(n)$ whence $G_{2^l k'}(n) = G_{2^\delta k'}(n)$ where $\delta = \delta(l) = 0$ if $2|l$ and $\delta = 1$ when $2\nmid l$.

Analogous to Proposition \ref{prop:powersoftwobdd}, we will prove that
\begin{align}\label{eqn:knot0goal}
\cal T(N_1, N_2; \alpha, t_1, t_2) 
&:= \sum_{(k, 2)=1 } \sumtwo_{(2a, n_1n_2)=1} \frac{\lambda_f(n_1) \lambda_f(n_2)}{n_1^{1/2+it_1} n_2^{1/2+it_2}} \frac{G_{2^{\delta}k}(n_1n_2)}{n_1 n_2} \check{J}\bfrac{kX\alpha}{2 n_1n_2}  G\bfrac{n_1}{N_1} G\bfrac{n_2}{N_2} \notag \\
&\ll \frac{a^\epsilon \sqrt{N_1 N_2} }{\alpha X} (1+|t_1|)^2(1+|t_2|)^2,
\end{align}for any $\alpha > 0$ and $\delta = 0, 1$.  Actually, following the proof of Proposition \ref{prop:powersoftwobdd} would give $(1+|t_i|)^{3/2+1/5}$ in place of $(1+|t_i|)^2$.  Then, to prove \eqref{eqn:kneq0}, we first apply \eqref{eqn:knot0goal} with $\alpha = \frac{2^{l}}{a^2}$ along with dyadic summation over $N_i$ and Mellin inversion to see that
\begin{align*}
T_l &\ll \sumd_{N_1, N_2}  \left(1+\frac{N_1}{\cal N}\right)^{-4} \left(1+\frac{N_2}{\cal N}\right)^{-4} \\
&\times \int_{-\infty}^\infty \int_{-\infty}^\infty  |T(N_1, N_2; 2^{l} a^{-2}, t_1, t_2)| (1+|t_1|)^{-10} (1+|t_2|)^{-10} dt_1 dt_2 \\
&\ll \frac{a^{2+\epsilon} \cal N}{2^l X},
\end{align*}and so the quantity in \eqref{eqn:kneq0} is $T_0 - \sum_{l \ge 1} T_{l}\ll \frac{a^{2+\epsilon} \cal N}{X}$.

The proof of \eqref{eqn:knot0goal} proceeds along the same lines as the proof of Proposition \ref{prop:powersoftwobdd} in the range $|k_1| \le \frac{N_1 N_2}{\alpha X}$ with two minor differences.  The first is that we now sum over $(k_1 k_2, 2) = 1$ with $k_1 k_2^2 =2^\delta k$ and $k_1$ squarefree.  When $\delta = 1$, we sum $2k_1$ over even squarefree numbers, while when $\delta=0$, we have $k_1$ running over odd squarefree numbers.  In both cases we may complete the sum to all squarefrees upon taking absolute values.  The $(k_2, 2) = 1$ condition changes $Z(\alpha, \beta, \gamma)$ by a benign factor of $1 - \frac{1}{2^{2\gamma}}$.  The second difference is that by Lemma \ref{lem:Z}, the condition $(2a, n_1n_2) = 1$ changes $Z(\alpha, \beta, \gamma)$ by a finite Euler product over $p|a$ which is ultimately bounded by $\ll d(a) \ll a^\epsilon$.

Now, suppose that $\cal K_1 \le |k_1| < 2\cal K_1$, with $\cal K_1 \ge \frac{N_1 N_2}{\alpha X}$.  Then the proof proceeds as before, with the only difference being that the integral in $s$ is moved to $\tRe s = 6/5$, while the integrals in $u$ and $v$ are still moved to $\tRe u, v = -1/2$.  This eventually gives a bound of
\begin{align*}
\ll a^\epsilon(\alpha X)^{-6/5} N_1^{6/5 - 1/2} N_2^{6/5 - 1/2} \cal K_1^{-1/5}(1+|t_1|)^2(1+|t_2|)^2,
\end{align*} 
and dyadic summation over $\cal K_1 \ge \frac{N_1 N_2}{\alpha X}$ gives the bound
\begin{align*}
&\ll a^\epsilon \bfrac{N_1 N_2}{\alpha X}^{6/5} (N_1N_2)^{-1/2} \bfrac{N_1 N_2}{\alpha X}^{-1/5} (1+|t_1|)^2(1+|t_2|)^2 \\
&\ll a^\epsilon \frac{\sqrt{N_1 N_2}}{\alpha X}(1+|t_1|)^2(1+|t_2|)^2
\end{align*}as desired.


\end{document}